\documentclass[reqno]{amsart}
\usepackage{amsfonts}
\usepackage{amsmath}
\usepackage{amssymb}
\usepackage{mathrsfs}
\usepackage[colorlinks]{hyperref}
\usepackage{amsmath,tikz}
\usetikzlibrary{matrix}
\usepackage{algorithm}
\usepackage{algpseudocode}  
\usepackage{algorithmicx}
\usepackage{multirow}
\usepackage{graphicx,lscape}
\usepackage[shortlabels]{enumitem}
\graphicspath{}

\begin{document}
\title[\hfil Backward problem for time-fractional pseudoparabolic equation]
{A backward problem for the time-fractional pseudo-parabolic equation with a variable coefficient}


\author[Arshyn Altybay]{Arshyn Altybay}
\address{
Arshyn Altybay: 
\endgraf
Institute of Mathematics and Mathematical Modeling,
\endgraf
 125 Pushkin str., 050010 Almaty, Kazakhstan
\endgraf
International Engineering and Technological University,
\endgraf
 89/21 Al-Farabi Avenue,  Almaty 050060, Kazakhstan
\endgraf
\endgraf  
    al--Farabi Kazakh National University
\endgraf
  71 al--Farabi ave., 050040 Almaty, Kazakhstan
  \endgraf
{\it E-mail address} {\rm arshyn.altybay@gmail.com, arshyn.altybay@math.kz}
}

\thanks{}

\begin{abstract}
This work addresses an inverse reconstruction task for a time-fractional pseudo-parabolic model with a temporally varying coefficient. By imposing Dirichlet boundary conditions, we aim to recover the unknown initial state from observations collected at the final time.

From a theoretical perspective, we derive existence and uniqueness results by proving that, under suitable hypotheses, the problem admits a unique solution. Computationally, we introduce a finite-difference discretisation based on a time-stepping strategy and provide a detailed stability and convergence analysis. Leveraging the resulting forward solver, we then formulate an initial-data identification procedure using Tikhonov regularisation. The proposed approach is validated with numerical simulations, and its resilience is assessed via experiments that incorporate perturbations in the final-time measurements.
\end{abstract}

\maketitle
\numberwithin{equation}{section} 
\newtheorem{theorem}{Theorem}[section]
\newtheorem{lemma}[theorem]{Lemma}
\newtheorem{corollary}[theorem]{Corollary}
\newtheorem{remark}[theorem]{Remark}
\newtheorem{definition}[theorem]
{Definition}
\newtheorem{proposition}[theorem]
{Proposition}
\allowdisplaybreaks


\section{Introduction}\label{sec:intro}

Pseudo-parabolic equations (Sobolev-type equations) are a mathematical model for a wide class of important physical and mechanical applications, such as filtration processes in porous media \cite{BER_1989}, population aggregation dynamics \cite{Pad_2004}, nonlinear long-wave propagation \cite{Ting_1963, BBM_1972}, second-order unsteady fluid flows \cite{Hui_1968}, fluid transport in fractured rocks \cite{BZK_1960}, and non-Newtonian fluid motion \cite{AKS_2011, ZVY_2010}. The inverse problems for a linear, semilinear and nonlinear pseudoparabolic equation, both by itself and with variable coefficients, have been extensively investigated; a representative selection of related works can be found in \cite{AA_1997}--\cite{Bockstal_2026}.

In this work, we study a backward problem for a time-fractional pseudo-parabolic equation
\begin{equation}\label{1.1}
\partial_t^\alpha u(x,t)-u_{xx}(x,t)-\mu(t)\,u_{xxt}(x,t)=f(x,t),
\qquad (x,t)\in(0,l)\times(0,T),
\end{equation}
subject to the Dirichlet boundary conditions
\begin{equation}\label{1.2}
u(0,t)=u(l,t)=0,\qquad t\in[0,T],
\end{equation}
and the final-time measurement
\begin{equation}\label{finalT}
u(x,T)=\psi(x),\qquad x\in[0,l].
\end{equation}
Here, $u(x,t)$ denotes the temperature, $f(x,t)$ is a given source term,
$\mu(t)$ is a pseudo-parabolic coefficient, and $\psi$
is the measured final-time state. The operator $\partial_t^\alpha$ denotes the Caputo fractional derivative of order $\alpha\in(0,1)$\cite{Kilbas06}. 
The backwards-in-time problem consists of reconstructing the initial state
\begin{equation}\label{initial}
u_0(x):=u(x,0),\qquad x\in[0,l],
\end{equation}
and, consequently, the whole trajectory $u$ on $[0,T]$ from the data $(f,\psi)$.

The equation \eqref{1.1} models a time-fractional pseudo-parabolic diffusion process with memory and relaxation effects, describing anomalous transport phenomena influenced by time-dependent material properties and external forcing in the space-time domain $(0,l)\times(0,T)$. 
In engineering applications, backwards-in-time problems are essential because they focus on reconstructing past physical states using data observed at a later time.

A pioneering contribution to the well-posedness analysis of the backward-in-time problem for the time-fractional diffusion equation is given in \cite{Sakamoto_2011}. Further developments for backward problems in fractional diffusion-wave models include theoretical and numerical studies in bounded domains \cite{WZ_2018}, as well as iterative regularisation approaches (e.g., Landweber-type methods) for identifying unknown initial data in time-space fractional settings \cite{YZL_2020}. General well-posedness frameworks for backwards-in-time problems for time-fractional diffusion and diffusion-wave equations were also established in \cite{Floridia_2020, Yamamoto_2020}. In addition, simultaneous recovery of multiple initial values for time-fractional diffusion–wave equations was investigated in \cite{ZVZ_2021}.

For fractional pseudo-parabolic equations, backward problems and related stability issues have been analysed in \cite{LKHC_2020, LZS21}. Regularisation methods for reconstructing the initial value in time-fractional pseudo-parabolic models were developed in \cite{YXJ_2022}, and robustness under random noise was further studied in \cite{DR_2023}. Related inverse initial data problems for time-fractional pseudo-hyperbolic equations were considered in \cite{ST_2024}.

Despite recent progress, backwards-in-time analysis for time-fractional pseudo-parabolic equations remains underexplored, particularly in terms of well-posedness theory, numerical reconstruction, and noisy data treatment.
To address this deficiency, we undertake a comprehensive theoretical and numerical investigation of the problem.

Compared with the constant-coefficient case, the present variable-coefficient problem has an essentially different spectral structure. If $\mu$ is constant, the Fourier modes may be studied by using coefficient-independent damping factors or special-function representations. In contrast, when $\mu=\mu(t)$, the $k$-th Fourier coefficient satisfies
\[
\partial_t^\alpha u_k(t)+\mu(t)\lambda_k u_k'(t)+\lambda_k u_k(t)=f_k(t),
\]
where the coefficient of $u_k'$ depends on time. Therefore, the backward
reconstruction cannot be reduced to a simple explicit multiplier formula.
Instead, each mode must be analysed through a weakly singular Volterra equation with a time-dependent kernel. The main technical point is to obtain estimates that are uniform with respect to the Fourier mode $k$, in particular, the positivity and lower boundedness of the reconstruction factor $\mathcal A_k(T)$.

The main contributions of this paper are as follows. First, we establish classical solvability of the backward problem under explicit regularity and compatibility assumptions on the final data and source term. Second, by reducing
the Fourier modes to Volterra equations with time-dependent kernels, we prove the positivity of the reconstruction factor and derive an $L^2$-stability estimate for the recovered initial state. Third, we construct a fully discrete
finite-difference solver based on the graded-mesh L1 approximation and central spatial differences, and combine it with Tikhonov regularisation for stable initial-state reconstruction from noisy final-time data.

The remainder of the manuscript is structured as follows:
Section~\ref{sec:wp} establishes the existence and uniqueness of the backward problem. Section~\ref{sec:direct_num} develops and analyses a stable numerical scheme for the direct model. Section~\ref{sec:inverse_num} presents the reconstruction procedure for identifying the unknown initial state from finite-time measurements. Numerical validations, including tests with noisy data, are given in Section~\ref{sec:numexp}.

\section{Existence and Uniqueness of the problem}\label{sec:wp}

In this section, we establish the existence and uniqueness of a solution to the backwards-in-time problem for the time fractional pseudo-parabolic equation \eqref{1.1} subject to the Dirichlet boundary conditions \eqref{1.2} and the final-time measurement \eqref{finalT}.

\subsection{Definition and assumptions}
\begin{definition}\label{def:frac}
A pair $(u,u_0)$ is called a \emph{classical solution} of \eqref{1.1}--\eqref{finalT} if
\[
u\in C\big([0,T];C([0,l])\big)\cap C^1\big((0,T);C^2([0,l])\big),\qquad
u_0\in C([0,l]),
\]
and, in addition,
\[
\partial_t^\alpha u\in C\big((0,T)\times[0,l]\big),
\]
where the Caputo derivative is defined by
\[
\partial_t^\alpha u(x,t)
=\frac{1}{\Gamma(1-\alpha)}\int_0^t (t-s)^{-\alpha}\,\partial_s u(x,s)\,ds,
\qquad (x,t)\in[0,l]\times(0,T).
\]
Moreover, $u(\cdot,0)=u_0$ on $[0,l]$, $u(0,t)=u(l,t)=0$ for $t\in[0,T]$,
$u(\cdot,T)=\psi$ on $[0,l]$, and \eqref{1.1} holds pointwise on $(0,l)\times(0,T)$.
\end{definition}

\begin{remark}
The condition $u_0(0)=u_0(l)=0$ is a compatibility condition inherited from
the homogeneous Dirichlet boundary condition $u(0,t)=u(l,t)=0$. Indeed, for a
classical solution continuous up to $t=0$, taking $t\to0^+$ in the boundary
condition gives $u(0,0)=u(l,0)=0$, and hence $u_0(0)=u_0(l)=0$. Thus, this
condition is not an additional measurement assumption but a natural
compatibility requirement for the classical solution concept.
\end{remark}

\textbf{Assumptions.}
Throughout this section, we assume:

\begin{enumerate}[\textbf{F}1,leftmargin=2.4cm]
\item \label{F1}
$\alpha\in(0,1)$, $\mu\in C([0,T])$, and there exists $\mu_0>0$ such that
$\mu(t)\ge \mu_0$ for all $t\in[0,T]$.

\item \label{F2}
$\psi\in H^3(0,l)$ and
\[
\psi(0)=\psi(l)=0,\qquad \psi''(0)=\psi''(l)=0.
\]

\item \label{F3}
$f\in L^2\big(0,T;H^3(0,l)\big)$ and for a.e. $t\in(0,T)$,
\[
f(0,t)=f(l,t)=0,\qquad f_{xx}(0,t)=f_{xx}(l,t)=0.
\]
\end{enumerate}

For uniform-in-time estimates of the time-derivative series, we shall also use:
\begin{enumerate}[\textbf{F}4,leftmargin=2.4cm]
\item \label{F4}
For every $\delta\in(0,T)$, $f\in L^\infty(\delta,T;H^3(0,l))$.
\end{enumerate}

\begin{remark}
The assumptions \textup{F2}--\textup{F4} are imposed in order to obtain a classical solution and to justify termwise differentiation of the Fourier series. In particular, the $H^3$-regularity and the compatibility conditions ensure uniform convergence of the series for $u$, $u_{xx}$, and, away from
$t=0$, the series for $u_t$, $u_{xxt}$, and $\partial_t^\alpha u$.
These assumptions are sufficient for the classical solvability result proved below. We do not claim that they are optimal. Under weaker assumptions, one may expect mild or weak solutions, but this is outside the scope of the present work.
\end{remark}

\subsection{Existence and uniqueness}
\begin{theorem}[Existence, uniqueness, and explicit reconstruction]\label{thm:wellposed_frac}
Let \ref{F1}--\ref{F4} hold. Then the inverse problem \eqref{1.1}--\eqref{finalT}
admits a unique classical solution $(u,u_0)$ in the sense of Definition~\ref{def:frac}.
Moreover, $(u,u_0)$ is given by
\begin{align}
u(x,t)
&=\sum_{k=1}^\infty
\Bigg[
\frac{\mathcal{A}_k(t)}{\mathcal{A}_k(T)}\,\psi_k
+\mathcal{B}_k(t)-\frac{\mathcal{A}_k(t)}{\mathcal{A}_k(T)}\,\mathcal{B}_k(T)
\Bigg]e_k(x),
\qquad (x,t)\in[0,l]\times[0,T],
\label{u_frac_compact_thm}\\[4pt]
u_0(x)
&=\sum_{k=1}^\infty
\Bigg[
\frac{\psi_k-\mathcal{B}_k(T)}{\mathcal{A}_k(T)}\Bigg]e_k(x),
\qquad x\in[0,l],
\label{phi_frac_compact_thm}
\end{align}
where
\[
e_k(x)=\sin\Big(\frac{k\pi}{l}x\Big),\qquad
\lambda_k=\Big(\frac{k\pi}{l}\Big)^2,\qquad k=1,2,\dots,
\]
$\psi_k$ and $f_k$ are defined in \eqref{coeff:psi:f}, and $\mathcal A_k,\mathcal B_k$ are defined by
\eqref{def:AkBk_frac} below. In particular, Lemma~\ref{lem:Ak_positive} yields $\mathcal A_k(T)>0$,
so the coefficients in \eqref{u_frac_compact_thm}--\eqref{phi_frac_compact_thm} are well-defined.

Furthermore:
\begin{enumerate}[(i)]
\item The series \eqref{u_frac_compact_thm} converges uniformly on $[0,l]\times[0,T]$ and
      the series for $u_{xx}$ converges uniformly on $[0,l]\times[0,T]$.
\item For every $\delta\in(0,T)$ the series for $u_t$, $u_{xxt}$ and
$\partial_t^\alpha u$ converge uniformly on $[0,l]\times[\delta,T]$. 
\end{enumerate}
\end{theorem}

\begin{proof}
\textbf{Formal Solution.}
Let $\{e_k\}_{k\ge1}$ be the sine eigenfunctions of $-\partial_{xx}$ on $(0,l)$:
\[
e_k(x)=\sin\Big(\frac{k\pi}{l}x\Big),\qquad -e_k''=\lambda_k e_k,\qquad
\lambda_k=\Big(\frac{k\pi}{l}\Big)^2.
\]
Then $\{e_k\}_{k\ge1}$ is an orthogonal basis of $L^2(0,l)$ with
\[
\int_0^l e_k(x)e_j(x)\,dx=\frac{l}{2}\delta_{kj}.
\]
For $v\in L^2(0,l)$ define its sine and cosine Fourier coefficients
\[
v_k:=\frac{2}{l}\int_0^l v(x)\sin\Big(\frac{k\pi}{l}x\Big)\,dx,\qquad
v_k^c:=\frac{2}{l}\int_0^l v(x)\cos\Big(\frac{k\pi}{l}x\Big)\,dx.
\]
Then Parseval's identity and Bessel's inequalities read
\begin{equation}\label{bessel_cos_frac}
\|v\|_{L^2(0,l)}^2=\frac{l}{2}\sum_{k=1}^\infty |v_k|^2,
\qquad
\sum_{k=1}^\infty |v_k|^2 \le \frac{2}{l}\|v\|_{L^2(0,l)}^2,\quad
\sum_{k=1}^\infty |v_k^c|^2 \le \frac{2}{l}\|v\|_{L^2(0,l)}^2 .
\end{equation}

Assume that $f$ and $\psi$ have sine-series expansions
\[
f(x,t)=\sum_{k=1}^{\infty}f_k(t)e_k(x),\qquad
\psi(x)=\sum_{k=1}^{\infty}\psi_k e_k(x),
\]
where $f_k(t)$ and  $\psi_k$ are the sine coefficients of $f$ and $\psi$ defined by 
\begin{equation}\label{coeff:psi:f}
f_k(t)=\frac{2}{l}\int_0^l f(x,t)\sin\Big(\frac{k\pi}{l}x\Big)\,dx, \quad \psi_k=\frac{2}{l}\int_0^l \psi(x)\sin\Big(\frac{k\pi}{l}x\Big)\,dx,  
\end{equation}
$k=1,2,\dots$, respectively.

We seek solution of \eqref{1.1}-\eqref{finalT} in the sine-series expansion form:
Expanding $u$ in the sine basis,
\[
u(x,t)=\sum_{k=1}^\infty u_k(t)e_k(x),\qquad e_k(x)=\sin\Big(\frac{k\pi}{l}x\Big),
\]
and using $u_{xx}=-\sum_{k\ge1}\lambda_k u_k e_k$ and $u_{xxt}=-\sum_{k\ge1}\lambda_k u_k' e_k$,
we obtain for each $k\ge1$ the mode equation
\begin{equation}\label{mode:eq_frac_thm}
\partial_t^\alpha u_k(t)+\mu(t)\lambda_k u_k'(t)+\lambda_k u_k(t)=f_k(t),
\qquad 0<t<T,\qquad u_k(T)=\psi_k.
\end{equation}

Let $y_k(t):=u_k'(t)$ and $u_{0,k}:=u_k(0)$. Then
\[
u_k(t)=u_{0,k}+\int_0^t y_k(s)\,ds,\qquad
\partial_t^\alpha u_k(t)=\frac{1}{\Gamma(1-\alpha)}\int_0^t (t-s)^{-\alpha}y_k(s)\,ds.
\]
Substituting into \eqref{mode:eq_frac_thm} yields the Volterra equation
\[
y_k(t)+\int_0^t H_k(t,s)y_k(s)\,ds
=\frac{f_k(t)-\lambda_ku_{0,k}}{\mu(t)\lambda_k},
\qquad 0<t<T,
\]
where
\[
H_k(t,s)=\frac{1}{\mu(t)}+\frac{1}{\mu(t)\lambda_k\Gamma(1-\alpha)}(t-s)^{-\alpha},
\qquad 0\le s<t\le T.
\]
Let $R_k$ be the corresponding resolvent kernel (given by the Neumann series). Then
\[
y_k(t)=g_k(t)+\int_0^t R_k(t,s)g_k(s)\,ds,\qquad
g_k(t):=\frac{f_k(t)-\lambda_ku_{0,k}}{\mu(t)\lambda_k}.
\]
Integrating in time and using $u_k(t)=u_{0,k}+\int_0^t y_k(\tau)\,d\tau$ yields
\begin{equation}\label{uk_AB_thm}
u_k(t)=u_{0,k}\,\mathcal A_k(t)+\mathcal B_k(t),\qquad 0\le t\le T,
\end{equation}
where
\begin{align}\label{def:AkBk_frac}
&\mathcal{A}_k(t)
:=1-\int_0^t \frac{1}{\mu(\tau)}\,d\tau
-\int_0^t\int_0^\tau R_k(\tau,s)\,\frac{1}{\mu(s)}\,ds\,d\tau,
\\&\mathcal{B}_k(t)
:=\int_0^t \frac{f_k(\tau)}{\mu(\tau)\lambda_k}\,d\tau
+\int_0^t\int_0^\tau R_k(\tau,s)\,\frac{f_k(s)}{\mu(s)\lambda_k}\,ds\,d\tau.
\end{align}

Evaluating \eqref{uk_AB_thm} at $t=T$ and using $u_k(T)=\psi_k$ gives
\[
\psi_k=u_{0,k}\mathcal A_k(T)+\mathcal B_k(T).
\]
By Lemma~\ref{lem:Ak_positive}, $\mathcal A_k(T)>0$ for all $k$, hence
\[
u_{0,k}=\frac{\psi_k-\mathcal B_k(T)}{\mathcal A_k(T)}.
\]
Substituting this into \eqref{uk_AB_thm} yields
\[
u_k(t)=\frac{\mathcal A_k(t)}{\mathcal A_k(T)}\psi_k
+\mathcal B_k(t)-\frac{\mathcal A_k(t)}{\mathcal A_k(T)}\mathcal B_k(T).
\]
Therefore, inserting into $u(x,t)=\sum_{k=1}^\infty u_k(t)e_k(x)$ and $u_0(x)=\sum_{k=1}^\infty u_{0,k} e_k(x)$
gives \eqref{u_frac_compact_thm}--\eqref{phi_frac_compact_thm}.

\textbf{Convergence and regularity.}
We justify that the series representations define a classical solution in the sense of
Definition~\ref{def:frac}. The argument is based on coefficient decay obtained by integration by parts,
together with Bessel's inequality \eqref{bessel_cos_frac} and the Weierstrass M-test.

\medskip
\noindent\emph{(a) Preparatory coefficient estimates.}
Under \ref{F2}, integrating by parts twice in the definition of $\psi_k$ gives
\begin{equation}\label{eq:psi_twoip_thm}
\psi_k=-\Big(\frac{l}{\pi k}\Big)^2(\psi'')_k,
\qquad
(\psi'')_k:=\frac{2}{l}\int_0^l \psi''(x)\sin\Big(\frac{k\pi}{l}x\Big)\,dx,
\end{equation}
and integrating once more yields
\begin{equation}\label{eq:psi_threeip_thm}
\psi_k=-\Big(\frac{l}{\pi k}\Big)^3(\psi''')_k^{c},
\qquad
(\psi''')_k^{c}:=\frac{2}{l}\int_0^l \psi'''(x)\cos\Big(\frac{k\pi}{l}x\Big)\,dx.
\end{equation}
Therefore, by Cauchy--Schwarz and Bessel's inequality \eqref{bessel_cos_frac},
\begin{align}
\sum_{k=1}^\infty |\psi_k|
&\le \frac{l^2}{\pi^2}\sum_{k=1}^\infty \frac{1}{k^2}|(\psi'')_k|
\le \frac{l^2}{\pi^2}\Big(\sum_{k=1}^\infty\frac{1}{k^4}\Big)^{1/2}
\Big(\sum_{k=1}^\infty|(\psi'')_k|^2\Big)^{1/2}
\le C\|\psi''\|_{L^2(0,l)}, \label{eq:sum_psik_abs_thm}\\[2mm]
\sum_{k=1}^\infty \lambda_k |\psi_k|
&= \Big(\frac{\pi}{l}\Big)^2 \sum_{k=1}^\infty k^2|\psi_k|
\le C \sum_{k=1}^\infty \frac{1}{k}|(\psi''')_k^{c}|
\le C\Big(\sum_{k=1}^\infty\frac{1}{k^2}\Big)^{1/2}
\Big(\sum_{k=1}^\infty|(\psi''')_k^{c}|^2\Big)^{1/2}
\le C\|\psi'''\|_{L^2(0,l)}. \label{eq:sum_lam_psik_thm}
\end{align}

Similarly, under \ref{F3}, integrating by parts twice in the definition of $f_k(t)$ gives, for a.e. $t\in(0,T)$,
\begin{equation}\label{eq:f_twoip_thm}
f_k(t)=-\Big(\frac{l}{\pi k}\Big)^2\,(f_{xx}(\cdot,t))_k,
\qquad
(f_{xx}(\cdot,t))_k:=\frac{2}{l}\int_0^l f_{xx}(x,t)\sin\Big(\frac{k\pi}{l}x\Big)\,dx.
\end{equation}
Hence, for a.e. $t$,
\begin{equation}\label{eq:sum_fk_pointwise_thm}
\sum_{k=1}^\infty |f_k(t)|
\le C\sum_{k=1}^\infty \frac{1}{k^2}|(f_{xx}(\cdot,t))_k|
\le C\|f_{xx}(\cdot,t)\|_{L^2(0,l)},
\end{equation}
and by Fubini and Cauchy--Schwarz in time,
\begin{align}
\sum_{k=1}^\infty \int_0^T |f_k(t)|\,dt
&\le C \int_0^T \|f_{xx}(\cdot,t)\|_{L^2(0,l)}\,dt
\le C\sqrt{T}\,\|f_{xx}\|_{L^2(0,T;L^2(0,l))} <\infty. \label{eq:sum_int_fk_thm}
\end{align}
Moreover, if \ref{F4} holds, then for every $\delta\in(0,T)$ we have
\begin{equation}\label{eq:sup_sum_fk_delta_thm}
\sup_{t\in[\delta,T]}\sum_{k=1}^\infty |f_k(t)|
\le C \sup_{t\in[\delta,T]}\|f_{xx}(\cdot,t)\|_{L^2(0,l)} <\infty.
\end{equation}

\medskip
\noindent\emph{(b) Uniform convergence of the series for $u$.}
From the explicit mode representation
\[
u_k(t)=\frac{\mathcal A_k(t)}{\mathcal A_k(T)}\psi_k
+\mathcal B_k(t)-\frac{\mathcal A_k(t)}{\mathcal A_k(T)}\mathcal B_k(T),
\]
and Lemma~\ref{lem:Ak_positive} (which yields $\mathcal A_k(T)>0$), we obtain a $k$-uniform bound as follows.
The Volterra resolvent construction and \ref{F1} imply that there exists $C>0$, independent of $k$, such that
\begin{equation}\label{eq:AkBk_uniform_bounds_thm}
\sup_{t\in[0,T]}|\mathcal A_k(t)|\le C,
\qquad
\sup_{t\in[0,T]}|\mathcal B_k(t)|
\le \frac{C}{\lambda_k}\int_0^T |f_k(s)|\,ds .
\end{equation}
Consequently,
\begin{equation}\label{eq:uk_sup_bound_thm}
\sup_{t\in[0,T]}|u_k(t)|
\le C\Big(|\psi_k|+\frac{1}{\lambda_k}\int_0^T |f_k(s)|\,ds\Big)
=:C M_k.
\end{equation}
Since $|e_k(x)|\le 1$ on $[0,l]$, it follows that
\[
|u_k(t)e_k(x)|\le C M_k,\qquad (x,t)\in[0,l]\times[0,T].
\]
Using \eqref{eq:sum_psik_abs_thm}, \eqref{eq:sum_int_fk_thm}, and $\sum_{k\ge1}\lambda_k^{-1}<\infty$, we get
\[
\sum_{k=1}^\infty M_k
\le \sum_{k=1}^\infty |\psi_k|
+\sum_{k=1}^\infty \frac{1}{\lambda_k}\int_0^T|f_k(s)|\,ds
<\infty.
\]
Therefore, by the Weierstrass M-test, the series \eqref{u_frac_compact_thm} converges uniformly on
$[0,l]\times[0,T]$. In particular, $u\in C([0,T];C([0,l]))$.

\medskip
\noindent\emph{(c) Uniform convergence of the series for $u_{xx}$.}
Differentiating formally in $x$ gives
\[
u_{xx}(x,t)=-\sum_{k=1}^\infty \lambda_k u_k(t)e_k(x).
\]
We justify uniform convergence by estimating the $k$-th term:
from \eqref{eq:uk_sup_bound_thm},
\begin{equation}\label{eq:lam_uk_sup_bound_thm}
\sup_{t\in[0,T]}\lambda_k|u_k(t)|
\le C\Big(\lambda_k|\psi_k|+\int_0^T|f_k(s)|\,ds\Big).
\end{equation}
The right-hand side is summable in $k$ by \eqref{eq:sum_lam_psik_thm} and \eqref{eq:sum_int_fk_thm}.
Hence
\[
\sum_{k=1}^\infty \sup_{t\in[0,T]}\lambda_k|u_k(t)|<\infty,
\]
and since $|e_k(x)|\le 1$, the Weierstrass M-test yields uniform convergence of the series for $u_{xx}$
on $[0,l]\times[0,T]$. Thus $u_{xx}\in C((0,T)\times[0,l])$.
\noindent

Moreover, by Cauchy--Schwarz and \eqref{eq:f_twoip_thm}, for a.e. $t\in(0,T)$,
\begin{equation}\label{eq:fk_sup_kminus2_bound}
|f_k(t)|
\le C\,\frac{1}{k^2}\,\|f_{xx}(\cdot,t)\|_{L^2(0,l)}.
\end{equation}
Indeed, by Cauchy--Schwarz and $|(f_{xx}(\cdot,t))_k|\le C\|f_{xx}(\cdot,t)\|_{L^2(0,l)}$ (Bessel),
the claim follows from \eqref{eq:f_twoip_thm}.
Consequently, if \ref{F4} holds then for every $\delta\in(0,T)$,
\begin{equation}\label{eq:sup_fk_kminus2_bound}
\sup_{t\in[\delta,T]}|f_k(t)|
\le C\,\frac{1}{k^2}\,\sup_{t\in[\delta,T]}\|f_{xx}(\cdot,t)\|_{L^2(0,l)}.
\end{equation}
In particular, $\sum_{k\ge1}\sup_{t\in[\delta,T]}|f_k(t)|<\infty$.

\medskip
\noindent\emph{(d) Uniform convergence of the series for $u_{xxt}$ on $[0,l]\times[\delta,T]$.}
Assume \ref{F4} and fix $\delta\in(0,T)$. Let $y_k(t)=u_k'(t)$. From the resolvent identity,
\[
y_k(t)=g_k(t)+\int_0^t R_k(t,s)g_k(s)\,ds,\qquad
g_k(t)=\frac{f_k(t)-\lambda_ku_{0,k}}{\mu(t)\lambda_k}.
\]
Standard resolvent (solution-operator) estimates for second-kind Volterra equations
with weakly singular kernels imply that for every $\delta\in(0,T)$,
\begin{equation}\label{eq:yk_sup_delta_bound_thm}
\sup_{t\in[\delta,T]}|y_k(t)|
\le C_\delta \,\|g_k\|_{L^\infty(\delta,,T)},
\end{equation}
where $C_\delta$ depends only on $(\delta,T,\alpha)$ and bounds on $\mu$, but not on $k$;
see, e.g., Becker~\cite{Becker11} (resolvent theory for weakly singular kernels).
The $k$-independence follows since $\mu(t)\ge\mu_0$ and $1/\lambda_k\le 1/\lambda_1$, so the kernel bounds entering the resolvent estimate are uniform in $k$.

Multiplying by $\lambda_k$ and using $\mu(t)\ge\mu_0$ gives
\begin{equation}\label{eq:lam_yk_sup_delta_bound_thm}
\sup_{t\in[\delta,T]}\lambda_k|u_k'(t)|
\le C_\delta\Big(\sup_{t\in[\delta,T]}|f_k(t)|+\lambda_k|u_{0,k}|\Big).
\end{equation}
Using $u_{0,k}=(\psi_k-\mathcal B_k(T))/\mathcal A_k(T)$ together with
\eqref{eq:AkBk_uniform_bounds_thm} yields
\begin{equation}\label{eq:lam_varphi_bound_thm}
\lambda_k|u_{0,k}|
\le C\Big(\lambda_k|\psi_k|+\int_0^T|f_k(s)|\,ds\Big).
\end{equation}
Combining \eqref{eq:lam_yk_sup_delta_bound_thm}--\eqref{eq:lam_varphi_bound_thm} we arrive at
\begin{equation}\label{eq:lam_ukprime_majorant_thm}
\sup_{t\in[\delta,T]}\lambda_k|u_k'(t)|
\le C_\delta\Big(\sup_{t\in[\delta,T]}|f_k(t)|+\lambda_k|\psi_k|
+\int_0^T|f_k(s)|\,ds\Big).
\end{equation}

We now show that the right-hand side is summable in $k$.
By \eqref{eq:sup_fk_kminus2_bound} we have
$\sum_{k\ge1}\sup_{t\in[\delta,T]}|f_k(t)|<\infty$.
Moreover, $\sum_{k\ge1}\lambda_k|\psi_k|<\infty$ by \eqref{eq:sum_lam_psik_thm}
and $\sum_{k\ge1}\int_0^T|f_k(s)|\,ds<\infty$ by \eqref{eq:sum_int_fk_thm}.
Therefore,
\[
\sum_{k=1}^\infty \sup_{t\in[\delta,T]}\lambda_k|u_k'(t)|<\infty.
\]
Since $|e_k(x)|\le1$, the Weierstrass M-test implies that the series
\[
u_{xxt}(x,t)=-\sum_{k=1}^\infty \lambda_k u_k'(t)e_k(x)
\]
converges uniformly on $[0,l]\times[\delta,T]$.
Hence $u_{xxt}\in C([0,l]\times[\delta,T])$.

\medskip
\noindent\emph{(e) Uniform convergence of the series for $u_t$ on $[0,l]\times[\delta,T]$.}
Assume \ref{F4} and fix $\delta\in(0,T)$. From \eqref{eq:lam_ukprime_majorant_thm},
\[
\sup_{t\in[\delta,T]}|u_k'(t)|
\le \frac{1}{\lambda_k}\sup_{t\in[\delta,T]}\lambda_k|u_k'(t)|
\le \frac{C_\delta}{\lambda_k}
\Big(\sup_{t\in[\delta,T]}|f_k(t)|+\lambda_k|\psi_k|+\int_0^T|f_k(s)|\,ds\Big),
\]
hence
\begin{equation}\label{eq:ukprime_sup_delta_bound}
\sup_{t\in[\delta,T]}|u_k'(t)|
\le C_\delta\Big(
|\psi_k|+\frac{1}{\lambda_k}\int_0^T|f_k(s)|\,ds
+\frac{1}{\lambda_k}\sup_{t\in[\delta,T]}|f_k(t)|
\Big).
\end{equation}

We show that the right-hand side is summable in $k$.
First, $\sum_{k\ge1}|\psi_k|<\infty$ by \eqref{eq:sum_psik_abs_thm}, and
\[
\sum_{k\ge1}\frac{1}{\lambda_k}\int_0^T|f_k(s)|\,ds<\infty
\]
by \eqref{eq:sum_int_fk_thm} together with $\sum_{k\ge1}\lambda_k^{-1}<\infty$.

Finally, using \eqref{eq:sup_fk_kminus2_bound} and $\lambda_k=(\pi k/l)^2$, we obtain
\[
\frac{1}{\lambda_k}\sup_{t\in[\delta,T]}|f_k(t)|
\le C\,\frac{1}{k^2}\cdot \frac{1}{k^2}\,
\sup_{t\in[\delta,T]}\|f_{xx}(\cdot,t)\|_{L^2(0,l)}
= C\,\frac{1}{k^4}\,\sup_{t\in[\delta,T]}\|f_{xx}(\cdot,t)\|_{L^2(0,l)}.
\]
Hence $\sum_{k\ge1}\lambda_k^{-1}\sup_{t\in[\delta,T]}|f_k(t)|<\infty$ since $\sum k^{-4}<\infty$.

Therefore,
\[
\sum_{k=1}^\infty \sup_{t\in[\delta,T]}|u_k'(t)|<\infty.
\]
Since $|e_k(x)|\le1$, the Weierstrass M-test implies that the series
\[
u_t(x,t)=\sum_{k=1}^\infty u_k'(t)e_k(x)
\]
converges uniformly on $[0,l]\times[\delta,T]$. In particular, $u_t\in C([0,l]\times[\delta,T])$.
Moreover, by part (d) the series for $u_{xxt}=-\sum_{k\ge1}\lambda_k u_k'(t)e_k(x)$ converges uniformly
on $[0,l]\times[\delta,T]$, so $u_t(\cdot,t)\in C^2([0,l])$ and $(u_t)_{xx}=u_{xxt}$ on $[\delta,T]$.

Since each $u_k(\cdot)e_k(\cdot)$ is $C^1$ in $t$ on $[\delta,T]$ and
$\sum_{k\ge1}u_k(t)e_k(x)$ converges uniformly on $[0,l]\times[\delta,T]$ (by (b)),
while $\sum_{k\ge1}u_k'(t)e_k(x)$ converges uniformly on $[0,l]\times[\delta,T]$ (by this step),
the standard theorem on termwise differentiation of uniformly convergent series on compact sets
implies that $u\in C^1([0,l]\times[\delta,T])$ and
\[
u_t(x,t)=\sum_{k=1}^\infty u_k'(t)e_k(x)\quad\text{on }[0,l]\times[\delta,T].
\]
Moreover, by (d) we have $(u_t)_{xx}=u_{xxt}$ on $[0,l]\times[\delta,T]$, hence
$u\in C^1([\delta,T];C^2([0,l]))$ for every $\delta\in(0,T)$, i.e.
$u\in C^1((0,T);C^2([0,l]))$.

\medskip
\noindent\emph{(f) Uniform convergence of the series for $\partial_t^\alpha u$ on $[0,l]\times[\delta,T]$.}
Fix $\delta\in(0,T)$ and assume \ref{F4}. By the mode equation \eqref{mode:eq_frac_thm}, for $t\in(0,T)$,
\[
\partial_t^\alpha u_k(t)=f_k(t)+\lambda_k u_k(t)+\mu(t)\lambda_k u_k'(t).
\]
Hence, for $t\in[\delta,T]$,
\begin{equation}\label{eq:caputo_series_thm}
\partial_t^\alpha u(x,t)
=\sum_{k=1}^\infty \partial_t^\alpha u_k(t)e_k(x)
=\sum_{k=1}^\infty \Big(f_k(t)+\lambda_k u_k(t)+\mu(t)\lambda_k u_k'(t)\Big)e_k(x).
\end{equation}

We show uniform convergence of the right-hand side on $[0,l]\times[\delta,T]$ termwise.
For the first series, by \eqref{eq:sup_fk_kminus2_bound} we have
$\sum_{k\ge1}\sup_{t\in[\delta,T]}|f_k(t)|<\infty$, and since $|e_k(x)|\le1$, the M-test yields
uniform convergence of $\sum_{k\ge1} f_k(t)e_k(x)$ on $[0,l]\times[\delta,T]$.

The second series $\sum_{k\ge1}\lambda_k u_k(t)e_k(x)$ equals $-u_{xx}(x,t)$ and is uniformly convergent on
$[0,l]\times[0,T]$ by part (c).
The third series $\sum_{k\ge1}\mu(t)\lambda_k u_k'(t)e_k(x)$ is uniformly convergent on
$[0,l]\times[\delta,T]$ by part (d) and boundedness of $\mu$ on $[\delta,T]$.
Therefore, \eqref{eq:caputo_series_thm} converges uniformly on $[0,l]\times[\delta,T]$, and in particular
$\partial_t^\alpha u\in C([0,l]\times[\delta,T])$.
Since $\delta\in(0,T)$ is arbitrary, we conclude $\partial_t^\alpha u\in C((0,T)\times[0,l])$.

\medskip
\noindent
Collecting (b)--(f), we obtain the stated regularity and the pointwise validity of \eqref{1.1}.

\textbf{Uniqueness.}
Let $(u_1,u_{0,1})$ and $(u_2,u_{0,2})$ be two classical solutions for the same $(f,\psi)$.
Set $u:=u_1-u_2$ and $\chi:=u_{0,1}-u_{0,2}$.
Then $u$ satisfies the homogeneous fractional pseudo-parabolic problem
\begin{equation}\label{eq:w_hom}
\partial_t^\alpha u - u_{xx}-\mu(t)\,u_{xxt}=0,\quad (x,t)\in(0,l)\times(0,T),
\end{equation}
with boundary conditions $u(0,t)=u(l,t)=0$, final-time measurement $u(x,T)=0$,
and initial state $u(x,0)=\chi(x)$.

Expand $u$ in the sine basis:
\[
u(x,t)=\sum_{k=1}^\infty u_k(t)e_k(x),\qquad
e_k(x)=\sin\Big(\frac{k\pi}{l}x\Big),\quad \lambda_k=\Big(\frac{k\pi}{l}\Big)^2.
\]
Then each coefficient $u_k$ solves (with $f_k\equiv0$)
\begin{equation}\label{eq:wk_hom}
\partial_t^\alpha u_k(t)+\mu(t)\lambda_k u_k'(t)+\lambda_k u_k(t)=0,
\qquad 0<t<T,\qquad u_k(T)=0,
\end{equation}
and $u_k(0)=\chi_k$ (the sine coefficient of $\chi$).

Repeating the derivation of the mode representation (Volterra reduction + resolvent),
but with $f_k\equiv0$, yields
\begin{equation}\label{eq:wk_Ak}
u_k(t)=\chi_k\,\mathcal A_k(t),\qquad 0\le t\le T,
\end{equation}
where $\mathcal A_k(t)$ is exactly the same functional defined in \eqref{def:AkBk_frac}.
Evaluating \eqref{eq:wk_Ak} at $t=T$ and using $u_k(T)=0$, we get
\[
0=u_k(T)=\chi_k\,\mathcal A_k(T).
\]
By  lemma \ref{lem:Ak_positive} we have $\mathcal A_k(T)\neq0$ for every $k$, hence $\chi_k=0$ for all $k$.
Therefore $u_k(t)\equiv0$ for all $k$ and all $t\in[0,T]$, so $u\equiv0$ on $(0,l)\times(0,T)$.
In particular, $\chi=u(\cdot,0)\equiv 0$, i.e. $u_{0,1}\equiv u_{0,2}$.

\end{proof}

\begin{lemma}[Positivity of $\mathcal A_k$]\label{lem:Ak_positive}
Let $\mathcal A_k$ be the function defined by \eqref{def:AkBk_frac}. Equivalently, $\mathcal A_k$ is the solution of the homogeneous mode problem
\begin{equation}\label{eq:Ak_mode}
\partial_t^\alpha \mathcal A_k(t)+\mu(t)\lambda_k\,\mathcal A_k'(t)+\lambda_k\,\mathcal A_k(t)=0,
\qquad t\in(0,T],\qquad \mathcal A_k(0)=1,
\end{equation}
where $\partial_t^\alpha$ is the Caputo derivative.

Then:
\begin{enumerate}[(i)]
\item $\mathcal A_k\in AC([0,T])$ and for every $\delta\in(0,T)$ one has
      $\mathcal A_k\in C^1([\delta,T])$ and $\partial_t^\alpha \mathcal A_k\in C([\delta,T])$.
      In particular, \eqref{eq:Ak_mode} holds pointwise on $[\delta,T]$ for every $\delta>0$.
\item $\mathcal A_k(t)>0$ for all $t\in[0,T]$. Hence $\mathcal A_k(T)>0$ for every $k\ge1$.
\end{enumerate}
\end{lemma}

\begin{proof}
Set $y_k(t):=\mathcal A_k'(t)$ and note that for $t\in[0,T]$,
\[
\mathcal A_k(t)=1+\int_0^t y_k(s)\,ds,
\qquad
\partial_t^\alpha \mathcal A_k(t)=\frac{1}{\Gamma(1-\alpha)}\int_0^t (t-s)^{-\alpha}y_k(s)\,ds.
\]
Substituting into \eqref{eq:Ak_mode} and dividing by $\mu(t)\lambda_k$ gives, for a.e. $t\in(0,T)$,
the Volterra equation of the second kind
\begin{equation}\label{eq:volterra_y_Ak}
y_k(t)+\int_0^t H_k(t,s)\,y_k(s)\,ds \;=\; -\frac{1}{\mu(t)},
\end{equation}
with a weakly singular kernel
\[
H_k(t,s)=\frac{1}{\mu(t)}+\frac{1}{\mu(t)\lambda_k\Gamma(1-\alpha)}(t-s)^{-\alpha},
\qquad 0\le s<t\le T.
\]
Since $\mu$ is continuous and bounded away from $0$ and $(t-s)^{-\alpha}$ is weakly singular with
$\alpha\in(0,1)$, the equation \eqref{eq:volterra_y_Ak} falls within the class of weakly singular
linear Volterra equations for which existence/uniqueness and resolvent-kernel representations are
available; see Becker~\cite{Becker11}. Consequently, $y_k$ admits a (locally) continuous
representative on $(0,T]$, and in particular $y_k\in C([\delta,T])$ for every $\delta\in(0,T)$.
Therefore $\mathcal A_k\in AC([0,T])$ and $\mathcal A_k\in C^1([\delta,T])$ for every $\delta>0$.
In particular, choosing any $\delta\in(0,T)$, all terms in \eqref{eq:Ak_mode} are continuous on
$[\delta,T]$ and thus \eqref{eq:Ak_mode} holds pointwise there.

Let $v\in AC([0,T])$ and let $t_0\in(0,T]$ be such that $v(t_0)=\min_{0\le s\le t_0} v(s)$.
For $AC$ functions, the Caputo derivative admits the identity
\[
\partial_t^\alpha v(t_0)
=\frac{1}{\Gamma(1-\alpha)}
\Bigg(
\frac{v(t_0)-v(0)}{t_0^\alpha}
+\alpha\int_0^{t_0}\frac{v(t_0)-v(s)}{(t_0-s)^{\alpha+1}}\,ds
\Bigg).
\]
Since $v(t_0)\le v(s)$ on $[0,t_0]$, the integral term is $\le0$, hence
\begin{equation}\label{eq:caputo_min_ineq_merged}
\partial_t^\alpha v(t_0)\le \frac{v(t_0)-v(0)}{\Gamma(1-\alpha)\,t_0^\alpha}.
\end{equation}

Assume, for contradiction, that $\mathcal A_k$ vanishes somewhere on $(0,T]$ and define the first
hitting time
\[
t_*:=\inf\{t\in(0,T]:\ \mathcal A_k(t)=0\}.
\]
Then $\mathcal A_k(t)>0$ for $t\in[0,t_*)$ and $\mathcal A_k(t_*)=0$.
Thus $t_*$ is a minimizer of $\mathcal A_k$ on $[0,t_*]$ with minimum $0$.
Applying \eqref{eq:caputo_min_ineq_merged} to $v=\mathcal A_k$ at $t_*$ yields
\[
\partial_t^\alpha \mathcal A_k(t_*)
\le \frac{\mathcal A_k(t_*)-\mathcal A_k(0)}{\Gamma(1-\alpha)\,t_*^\alpha}
=\frac{0-1}{\Gamma(1-\alpha)\,t_*^\alpha}<0.
\]
On the other hand, since $t_*>0$, choose $\delta\in(0,t_*)$. By Step~1, $\mathcal A_k\in C^1([\delta,T])$,
so $\mathcal A_k'(t_*)$ exists and \eqref{eq:Ak_mode} holds pointwise at $t=t_*$.
Using $\mathcal A_k(t_*)=0$ in \eqref{eq:Ak_mode} gives
\begin{equation}\label{eq:caputo_ode_at_tstar_merged}
\partial_t^\alpha \mathcal A_k(t_*)=-\mu(t_*)\lambda_k\,\mathcal A_k'(t_*).
\end{equation}
Moreover, $\mathcal A_k(t)>0$ for $t<t_*$ and $\mathcal A_k(t_*)=0$ implies $\mathcal A_k'(t_*)\le0$
(otherwise, if $\mathcal A_k'(t_*)>0$, then $\mathcal A_k(t)<0$ for $t<t_*$ sufficiently close to $t_*$,
contradicting the definition of $t_*$). Since $\mu(t_*)\ge\mu_0>0$ and $\lambda_k>0$, the right-hand side
of \eqref{eq:caputo_ode_at_tstar_merged} is $\ge0$, hence $\partial_t^\alpha \mathcal A_k(t_*)\ge0$.
This contradicts $\partial_t^\alpha \mathcal A_k(t_*)<0$ above.

Therefore $\mathcal A_k$ cannot hit zero on $(0,T]$. Since $\mathcal A_k(0)=1$, we conclude
$\mathcal A_k(t)>0$ for all $t\in[0,T]$, in particular $\mathcal A_k(T)>0$.
\end{proof}

\subsection{Continuous dependence on the data}\label{subsec:stab} In this subsection, we prove that, under assumptions \ref{F1}--\ref{F3}, the reconstruction of the initial state $u_0(x)=u(x,0)$ from the final-time measurement $\psi(x)=u(x,T)$ continuously depends on the data.

\begin{proposition}[Stability estimate for the reconstructed initial state]\label{prop:stability_phi_L2}
Assume \ref{F1}--\ref{F3}. Let $(u,u_0)$ be the solution of the  problem \eqref{1.1}--\eqref{finalT} given by \eqref{u_frac_compact_thm}--\eqref{phi_frac_compact_thm}.
Then there exists a constant $C>0$ (depending only on $\alpha,T,\mu_0,l$ and bounds on $\mu$) such that
\begin{equation}\label{eq:stab_phi_L2_generic}
\|u_0\|_{L^2(0,l)}
\le
C\Bigl(
\|\psi\|_{L^2(0,l)}+\sqrt{T}\,\|f\|_{L^2\!\left(0,T;L^2(0,l)\right)}
\Bigr).
\end{equation}
Moreover, for two data sets $(\psi_1,f_1)$ and $(\psi_2,f_2)$ with corresponding reconstructions
$u_{0,1},u_{0,2}$, one has
\begin{equation}\label{eq:stab_phi_diff_L2_generic}
\|u_{0,1}-u_{0,2}\|_{L^2(0,l)}
\le
C\Bigl(
\|\psi_1-\psi_2\|_{L^2(0,l)}+\sqrt{T}\,\|f_1-f_2\|_{L^2\!\left(0,T;L^2(0,l)\right)}
\Bigr).
\end{equation}
\end{proposition}
\begin{proof}
Recall that
\[
u_0(x)=\sum_{k=1}^\infty u_{0,k} e_k(x),
\qquad
u_{0,k}=\frac{\psi_k-\mathcal B_k(T)}{\mathcal A_k(T)},
\qquad
\lambda_k=\Big(\frac{k\pi}{l}\Big)^2.
\]
\textit{A uniform lower bound for $\mathcal A_k(T)$}.
For $\mathcal A_k$ we have (Lemma~\ref{lem:Ak_positive}) $\mathcal A_k(t)>0$ on $[0,T]$ and
$\mathcal A_k\in AC([0,T])$.
From the Volterra reduction \eqref{eq:volterra_y_Ak} with negative right-hand side and positive kernel,
the solution $y_k=\mathcal A_k'$ is nonpositive a.e., hence $\mathcal A_k$ is nonincreasing.
Therefore, for $t\in(0,T)$,
\[
\partial_t^\alpha \mathcal A_k(t)
=\frac{1}{\Gamma(1-\alpha)}\int_0^t (t-s)^{-\alpha}\,\mathcal A_k'(s)\,ds \le 0.
\]
Using the mode equation \eqref{eq:Ak_mode},
\[
\mu(t)\lambda_k\,\mathcal A_k'(t)+\lambda_k\,\mathcal A_k(t)=-\partial_t^\alpha \mathcal A_k(t)\ge 0,
\]

hence $\mathcal A_k'(t)\ge -\mu(t)^{-1}\mathcal A_k(t)$ for a.e. $t\in(0,T)$.
By Gr\"onwall's inequality,
\begin{equation}\label{eq:Ak_lower_bound}
\mathcal A_k(T)\ge \exp\!\Bigl(-\int_0^T \frac{1}{\mu(s)}\,ds\Bigr)\ge e^{-T/\mu_0}.
\end{equation}
In particular, $\sup_{k\ge1}\mathcal A_k(T)^{-1}\le e^{T/\mu_0}$.

\textit{A bound for $\mathcal B_k(T)$}.
From the uniform estimate \eqref{eq:AkBk_uniform_bounds_thm} (evaluated at $t=T$) there exists a constant $C_*>0$ (independent of $k$ and of the data) such that
\begin{equation}\label{eq:BkT_bound}
|\mathcal B_k(T)|
\le \frac{C_*}{\lambda_k}\int_0^T |f_k(s)|\,ds.
\end{equation}

\textit{Estimate of the Fourier coefficients $u_{0,k}$}.
Combining \eqref{eq:Ak_lower_bound}--\eqref{eq:BkT_bound} yields
\[
|u_{0,k}|
\le \frac{1}{\mathcal A_k(T)}\Bigl(|\psi_k|+|\mathcal B_k(T)|\Bigr)
\le e^{T/\mu_0}\Bigl(|\psi_k|+\frac{C_*}{\lambda_k}\int_0^T |f_k(s)|\,ds\Bigr).
\]
Using $(a+b)^2\le 2a^2+2b^2$ and Cauchy--Schwarz in time,
\[
\Bigl(\int_0^T |f_k(s)|\,ds\Bigr)^2 \le T\int_0^T |f_k(s)|^2\,ds,
\qquad
\frac{1}{\lambda_k^2}\le \frac{1}{\lambda_1^2}.
\]

By Parseval,
\[
\|u_0\|_{L^2(0,l)}^2=\frac{l}{2}\sum_{k\ge1}|u_{0,k}|^2,
\quad
\|\psi\|_{L^2(0,l)}^2=\frac{l}{2}\sum_{k\ge1}|\psi_k|^2,
\quad
\|f\|_{L^2(0,T;L^2(0,l))}^2=\frac{l}{2}\sum_{k\ge1}\int_0^T |f_k(t)|^2\,dt.
\]
Hence,
\[
\|u_0\|_{L^2(0,l)}^2
\le C\Bigl(\|\psi\|_{L^2(0,l)}^2 + T\,\|f\|_{L^2(0,T;L^2(0,l))}^2\Bigr),
\]
which implies \eqref{eq:stab_phi_L2_generic} (after taking square roots and adjusting $C$).

Finally, \eqref{eq:stab_phi_diff_L2_generic} follows by applying the same argument to the difference
data $(\psi_1-\psi_2,f_1-f_2)$, since the reconstruction is linear in $(\psi,f)$.
\end{proof}


\section{Finite difference approximation of the direct problem}\label{sec:direct_num}
In this section, we describe the finite-difference discretisation used in the implementation for the
time-fractional pseudo-parabolic direct problem
\begin{equation}\label{eq:direct_frac}
\partial_t^\alpha u(x,t)-u_{xx}(x,t)-\mu(t)\,u_{xxt}(x,t)=f(x,t),
\qquad (x,t)\in(0,l)\times(0,T),
\end{equation}
subject to the homogeneous Dirichlet boundary conditions $u(0,t)=u(l,t)=0$.
We assume that $\mu\in C([0,T])$ and $f$ is given. The Caputo derivative is discretised by an $L1$-type scheme
on a graded temporal mesh, while the spatial derivatives are approximated by second-order central
differences. The resulting method leads to a tridiagonal linear system at each time level\cite{Altybay_26}.

\subsection{Discretisations}\label{subsec:disc}
In this subsection, we describe spatial and temporal discretisations and the full discrete scheme.

\textbf{Spatial discretisation.} Let $x_i=ih$ for $i=0,\dots,N$, where $h=l/N$. For a grid vector
$\mathbf{v}=(v_1,\dots,v_{N-1})^\top\in\mathbb{R}^{N-1}$ we enforce the boundary conditions by setting
$v_0=v_N=0$ and define the standard second-order discrete Laplacian $L_h:\mathbb{R}^{N-1}\to\mathbb{R}^{N-1}$ by
\begin{equation}\label{eq:lap_frac}
(L_h\mathbf{v})_i=\frac{v_{i-1}-2v_i+v_{i+1}}{h^2},\qquad i=1,\dots,N-1.
\end{equation}
The matrix representation of $L_h$ is symmetric tridiagonal and $(-L_h)$ is symmetric positive definite on
$\mathbb{R}^{N-1}$.

We denote the interior nodal vector of the numerical solution by
\[
\mathbf{u}^k:=(u_1^k,\dots,u_{N-1}^k)^\top,\qquad u_i^k\approx u(x_i,t_k),
\]
where $\{t_k\}_{k=0}^M$ is the temporal grid defined below.

\textbf{L1 approximation on graded meshes.}\label{subsec:caputo_L1}
Let $0=t_0<t_1<\cdots<t_M=T$ be a time grid. In the implementation, we allow graded meshes of the form
\begin{equation}\label{eq:graded_mesh}
t_k=T\Bigl(\frac{k}{M}\Bigr)^r,\,\ k=0,1,\dots,M,\,\, r\ge 1(\text{uniform mesh corresponds to } r=1),
\end{equation}
with $\tau_k:=t_k-t_{k-1}$ and $\tau_{\max}:=\max_{1\le k\le M}\tau_k$.

For each $k\ge1$, we approximate the Caputo derivative at 
$t_k$  using the $L1$ scheme
\begin{equation}\label{eq:L1_caputo}
\partial_t^\alpha u(x_i,t_k)
\approx
\frac{1}{\Gamma(1-\alpha)}
\sum_{j=1}^{k}
\int_{t_{j-1}}^{t_j}\frac{u_t(x_i,s)}{(t_k-s)^\alpha}\,ds
\;\approx\;
\sum_{j=1}^{k} w_{k,j}\,\bigl(u_i^j-u_i^{j-1}\bigr),
\end{equation}
where the weights $w_{k,j}$ depend on the mesh.
For a nonuniform grid, they can be written as
\begin{equation}\label{eq:L1_weights}
w_{k,j}
=
\frac{1}{\Gamma(2-\alpha)}\,d_{k,j},
\quad
d_{k,j}
=
\frac{(t_k-t_{j-1})^{1-\alpha}-(t_k-t_j)^{1-\alpha}}{t_j-t_{j-1}},
\quad 1\le j\le k.
\end{equation}

\textbf{Fully discrete scheme.}\label{subsec:fully_discrete}
Using $u_{xxt}=(u_{xx})_t$ and the discrete Laplacian $L_h$, we approximate at time $t_k$:
\[
u_{xx}(x_i,t_k)\approx (L_h\mathbf{u}^k)_i,
\qquad
u_{xxt}(x_i,t_k)\approx \frac{(L_h\mathbf{u}^k)_i-(L_h\mathbf{u}^{k-1})_i}{\tau_k}.
\]
Let $\boldsymbol{\mu}^k:=\mu(t_k)$ and
$
\mathbf{f}^k:=\bigl(f(x_1,t_k),\dots,f(x_{N-1},t_k)\bigr)^\top.
$
Combining \eqref{eq:L1_caputo}--\eqref{eq:L1_weights} with the above spatial discretisations yields, for $k=1,\dots,M$,
the linear system
\begin{equation}\label{eq:scheme_matrix_form}
\Bigl(\frac{d_{k,k}}{\Gamma(2-\alpha)}I - \Bigl(1+\frac{\mu^k}{\tau_k}\Bigr)L_h\Bigr)\mathbf{u}^k
=
\mathbf{r}^k
+
\mathbf{f}^k
-
\frac{\mu^k}{\tau_k}L_h\mathbf{u}^{k-1},
\end{equation}
where $\mathbf{r}^k$ collects the history (memory) contributions of the $L1$ approximation,
\begin{equation}\label{eq:history_term}
\mathbf{r}^k
=
\frac{d_{k,1}}{\Gamma(2-\alpha)}\mathbf{u}^0
+
\frac{1}{\Gamma(2-\alpha)}
\sum_{j=1}^{k-1}\bigl(d_{k,j+1}-d_{k,j}\bigr)\mathbf{u}^j .
\end{equation}
Here $d_{k,j}$ are defined in \eqref{eq:L1_weights}. The left-hand side matrix in \eqref{eq:scheme_matrix_form} is
tridiagonal because $L_h$ is tridiagonal. Therefore, at each time step the system is solved efficiently by the
Thomas algorithm.

Given an initial state $u_0(x)=u(x,0)$, we set
\[
\boldsymbol{u_{0,h}}
:=
\bigl(u_0(x_1),\dots,u_0(x_{N-1})\bigr)^\top\in\mathbb{R}^{N-1},
\qquad
\mathbf{u}^0=\boldsymbol{u_{0,h}}.
\]
The direct solver \eqref{eq:scheme_matrix_form}--\eqref{eq:history_term} then produces the discrete trajectory
$\{\mathbf{u}^k\}_{k=0}^M$ and, in particular, the terminal vector $\mathbf{u}^M$ used in the inverse step.

\subsection{Stability}\label{subsec:stab_conv}
In this subsection, we study the stability of the fully discrete method defined by
\eqref{eq:scheme_matrix_form}--\eqref{eq:history_term}. The Caputo derivative is approximated by the graded-mesh $L1$ formula, with the weights $d_{k,j}$ given in \eqref{eq:L1_weights}. Throughout this subsection, we assume that
\begin{equation}\label{eq:ass_mu}
\mu\in C([0,T]),
\qquad
0<\mu_0\le \mu(t)\le \mu_{\max},
\qquad t\in[0,T].
\end{equation}
We use the discrete inner product and norms introduced earlier. In
particular,
\[
(-L_h\mathbf v,\mathbf v)_h
=
\|\nabla_h\mathbf v\|_h^2.
\]

\textbf{Preliminary properties of the graded-mesh $L1$ operator.}
Define the graded-mesh $L1$ operator componentwise by
\begin{equation}\label{eq:disc_caputo_def}
\delta_t^\alpha\mathbf v^k
:=
\frac{1}{\Gamma(2-\alpha)}
\sum_{j=1}^{k}
d_{k,j}\bigl(\mathbf v^j-\mathbf v^{j-1}\bigr),
\qquad k\ge1,
\end{equation}
where $d_{k,j}$ are defined in \eqref{eq:L1_weights}. For every admissible temporal mesh, and in particular for the graded mesh considered here, one has
\begin{equation}\label{eq:d_props}
d_{k,j}>0,
\qquad
d_{k,1}\le d_{k,2}\le\cdots\le d_{k,k},
\qquad 1\le j\le k.
\end{equation}

\textbf{Solvability at each time level.}
Let $\tau_k=t_k-t_{k-1}$ and $\mu^k=\mu(t_k)$. The time-stepping
system \eqref{eq:scheme_matrix_form} can be written as
\begin{equation}\label{eq:Ak_system}
\mathbf A^k\mathbf u^k=\mathbf b^k,
\qquad
\mathbf A^k
=
\frac{d_{k,k}}{\Gamma(2-\alpha)}I
-
\left(1+\frac{\mu^k}{\tau_k}\right)L_h,
\end{equation}
where $\mathbf b^k$ contains the history term
\eqref{eq:history_term}, the source $\mathbf f^k$, and the contribution
\[
-\frac{\mu^k}{\tau_k}L_h\mathbf u^{k-1}.
\]

\begin{lemma}[Unique solvability]\label{lem:SPD_graded}
Under assumption \eqref{eq:ass_mu}, the matrix $\mathbf A^k$ is
symmetric positive definite on $\mathbb R^{N-1}$ for every $k\ge1$.
Consequently, the system \eqref{eq:Ak_system} has a unique solution.
\end{lemma}

\begin{proof}
For any $\mathbf v\ne\mathbf0$,
\[
(\mathbf A^k\mathbf v,\mathbf v)_h
=
\frac{d_{k,k}}{\Gamma(2-\alpha)}
\|\mathbf v\|_h^2
+
\left(1+\frac{\mu^k}{\tau_k}\right)
\|\nabla_h\mathbf v\|_h^2.
\]
Since $d_{k,k}>0$, $\mu^k>0$, and $-L_h$ is symmetric positive
definite, the right-hand side is strictly positive.
\end{proof}

Let $ \mathcal A_h:=-L_h.$
Since $\mathcal A_h$ is symmetric positive definite, there exists an
orthonormal basis
$\{\boldsymbol\varphi_q\}_{q=1}^{N-1}$, with respect to
$(\cdot,\cdot)_h$, such that
\[
\mathcal A_h\boldsymbol\varphi_q
=
\lambda_{q,h}\boldsymbol\varphi_q,
\qquad
0<\lambda_{1,h}\le\lambda_{2,h}\le\cdots\le\lambda_{N-1,h}.
\]
The discrete Poincar\'e inequality implies
\begin{equation}\label{eq:lambda_lower_bound}
\lambda_{1,h}^{-1/2}\le C_P,
\end{equation}
where $C_P$ is independent of $h$.

\begin{theorem}[Unconditional stability]\label{thm:stab_graded}
Let \eqref{eq:ass_mu} hold, and let
$\{\mathbf u^k\}_{k=0}^{M}$ be generated by
\eqref{eq:scheme_matrix_form}--\eqref{eq:history_term}. Then there exists
a constant $C>0$, depending only on
$
T,\,\, \mu_0,\,\, \mu_{\max},\,\, C_P,
$
but independent of $h$, $M$, and the time-step sizes $\tau_k$, such
that, for every $m\le M$,
\begin{align}
&\max_{0\le k\le m}\|\mathbf u^k\|_h^2
+
\max_{0\le k\le m}
\mu^k\|\nabla_h\mathbf u^k\|_h^2
+
\sum_{k=1}^{m}
\tau_k\|\nabla_h\mathbf u^k\|_h^2
\notag\\
&\qquad\le
C\left(
\|\mathbf u^0\|_h^2
+
\mu^0\|\nabla_h\mathbf u^0\|_h^2
+
\sum_{k=1}^{m}
\tau_k\|\mathbf f^k\|_h^2
\right),
\label{eq:stab_graded}
\end{align}
where $\mu^0=\mu(0)$.
\end{theorem}

\begin{proof}
Expand the numerical solution and the source in the eigenbasis of
$\mathcal A_h$:
\[
\mathbf u^k
=
\sum_{q=1}^{N-1}\widehat u_q^k\boldsymbol\varphi_q,
\qquad
\mathbf f^k
=
\sum_{q=1}^{N-1}\widehat f_q^k\boldsymbol\varphi_q.
\]
Introduce
\[
a_{k,j}
:=
\frac{d_{k,j}}{\Gamma(2-\alpha)}.
\]
For each eigenmode $q$, the fully discrete scheme takes the form
\begin{align}
&\left[
a_{k,k}
+
\left(1+\frac{\mu^k}{\tau_k}\right)\lambda_{q,h}
\right]\widehat u_q^k
\notag\\
&\quad=
a_{k,1}\widehat u_q^0
+
\sum_{j=1}^{k-1}
\bigl(a_{k,j+1}-a_{k,j}\bigr)\widehat u_q^j
+
\frac{\mu^k}{\tau_k}
\lambda_{q,h}\widehat u_q^{k-1}
+
\widehat f_q^k .
\label{eq:modal_scheme_stability}
\end{align}
Set
\[
B_{k,q}
:=
a_{k,k}
+
\left(1+\frac{\mu^k}{\tau_k}\right)\lambda_{q,h}.
\]
By \eqref{eq:d_props}, all coefficients of the previous time levels in
\eqref{eq:modal_scheme_stability} are nonnegative. Moreover,
\begin{align*}
&a_{k,1}
+
\sum_{j=1}^{k-1}
\bigl(a_{k,j+1}-a_{k,j}\bigr)
+
\frac{\mu^k}{\tau_k}\lambda_{q,h}
\\
&\qquad=
a_{k,k}
+
\frac{\mu^k}{\tau_k}\lambda_{q,h}
=
B_{k,q}-\lambda_{q,h}
<
B_{k,q}.
\end{align*}
Therefore, with
\[
M_q^{k-1}
:=
\max_{0\le j\le k-1}|\widehat u_q^j|,
\]
equation \eqref{eq:modal_scheme_stability} gives
\[
|\widehat u_q^k|
\le
\frac{B_{k,q}-\lambda_{q,h}}{B_{k,q}}M_q^{k-1}
+
\frac{|\widehat f_q^k|}{B_{k,q}}
\le
M_q^{k-1}
+
\frac{|\widehat f_q^k|}{B_{k,q}}.
\]
Since
\[
B_{k,q}
\ge
\frac{\mu^k}{\tau_k}\lambda_{q,h}
\ge
\frac{\mu_0}{\tau_k}\lambda_{q,h},
\]
we obtain
\[
|\widehat u_q^k|
\le
M_q^{k-1}
+
\frac{\tau_k}{\mu_0\lambda_{q,h}}
|\widehat f_q^k|.
\]
An induction over $k$ yields
\begin{equation}\label{eq:modal_stability_bound}
|\widehat u_q^k|
\le
|\widehat u_q^0|
+
\frac{1}{\mu_0\lambda_{q,h}}
\sum_{n=1}^{k}\tau_n|\widehat f_q^n|,
\qquad 1\le k\le M.
\end{equation}

Using Parseval's identity, Minkowski's inequality, and
\eqref{eq:lambda_lower_bound}, we obtain
\begin{align*}
\|\mathbf u^k\|_h
&\le
\|\mathbf u^0\|_h
+
\frac{1}{\mu_0\lambda_{1,h}}
\sum_{n=1}^{k}\tau_n\|\mathbf f^n\|_h
\\
&\le
\|\mathbf u^0\|_h
+
\frac{C_P^2\sqrt{T}}{\mu_0}
\left(
\sum_{n=1}^{k}
\tau_n\|\mathbf f^n\|_h^2
\right)^{1/2}.
\end{align*}
Consequently,
\begin{equation}\label{eq:L2_stability_intermediate}
\max_{0\le k\le m}\|\mathbf u^k\|_h^2
\le
C\left(
\|\mathbf u^0\|_h^2
+
\sum_{n=1}^{m}
\tau_n\|\mathbf f^n\|_h^2
\right).
\end{equation}

Similarly, multiplying \eqref{eq:modal_stability_bound} by
$\lambda_{q,h}^{1/2}$, summing over $q$, and again using Parseval's
identity gives
\begin{align*}
\|\nabla_h\mathbf u^k\|_h
&\le
\|\nabla_h\mathbf u^0\|_h
+
\frac{1}{\mu_0\lambda_{1,h}^{1/2}}
\sum_{n=1}^{k}\tau_n\|\mathbf f^n\|_h
\\
&\le
\|\nabla_h\mathbf u^0\|_h
+
\frac{C_P\sqrt{T}}{\mu_0}
\left(
\sum_{n=1}^{k}
\tau_n\|\mathbf f^n\|_h^2
\right)^{1/2}.
\end{align*}
Therefore,
\begin{equation}\label{eq:H1_stability_intermediate}
\max_{0\le k\le m}
\|\nabla_h\mathbf u^k\|_h^2
\le
C\left(
\|\nabla_h\mathbf u^0\|_h^2
+
\sum_{n=1}^{m}
\tau_n\|\mathbf f^n\|_h^2
\right).
\end{equation}
Since $\mu^k\le\mu_{\max}$ and
$\mu^0\ge\mu_0$, it follows that
\[
\max_{0\le k\le m}
\mu^k\|\nabla_h\mathbf u^k\|_h^2
\le
C\left(
\mu^0\|\nabla_h\mathbf u^0\|_h^2
+
\sum_{n=1}^{m}
\tau_n\|\mathbf f^n\|_h^2
\right).
\]
Finally,
\[
\sum_{k=1}^{m}
\tau_k\|\nabla_h\mathbf u^k\|_h^2
\le
T\max_{0\le k\le m}
\|\nabla_h\mathbf u^k\|_h^2.
\]
Combining this estimate with
\eqref{eq:L2_stability_intermediate} and
\eqref{eq:H1_stability_intermediate} proves
\eqref{eq:stab_graded}.
\end{proof}

\subsection{Convergence}\label{subsubsec:conv_graded}
In this subsection, we establish an error estimate for the fully discrete
scheme \eqref{eq:scheme_matrix_form}--\eqref{eq:history_term} on the
graded temporal mesh introduced in Subsection~\ref{subsec:disc}.

For the graded mesh \eqref{eq:graded_mesh}, the time-step sizes are
nondecreasing, and hence
\[
\tau_{\max}=\tau_M
=
T\left[
1-\left(1-\frac1M\right)^r
\right]
\le \frac{rT}{M}.
\]
Thus, for every fixed $r\ge1$,
\begin{equation}\label{eq:tau_max_graded}
\tau_{\max}=O(M^{-1}).
\end{equation}

Let $R_h$ denote the nodal restriction operator
\[
R_hv:=
\bigl(v(x_1),\dots,v(x_{N-1})\bigr)^\top.
\]
For the exact solution $u$, define
\[
\mathbf U^k:=R_hu(\cdot,t_k),
\qquad k=0,\dots,M.
\]
Let $\{\mathbf u^k\}_{k=0}^{M}$ be the numerical solution generated by
\eqref{eq:scheme_matrix_form}--\eqref{eq:history_term}, and define the
nodal error by
\[
\mathbf e^k:=\mathbf U^k-\mathbf u^k,
\qquad k=0,\dots,M.
\]
Since the exact initial data are imposed in the numerical scheme,
\begin{equation}\label{eq:error_initial_zero}
\mathbf e^0=\mathbf0.
\end{equation}

\textbf{Additional regularity assumptions.} Theorem~\ref{thm:wellposed_frac} establishes classical solvability under assumptions \ref{F1}--\ref{F4}. The convergence analysis below is conditional on additional space--time regularity of the exact solution. These stronger assumptions are not asserted to follow from \ref{F1}--\ref{F4}; they are imposed in order to obtain uniform consistency estimates for the central-difference approximation, the backward-difference approximation of $u_{xxt}$, and the $L1$ approximation of the Caputo derivative.

We assume that \eqref{eq:ass_mu} holds and that \begin{equation}\label{eq:reg_assumption_smooth} u\in C^1\bigl([0,T];H^4(0,l)\bigr) \cap C^2\bigl([0,T];H^3(0,l)\bigr), \qquad u(\cdot,t)\in H_0^1(0,l) \quad\text{for all }t\in[0,T], \end{equation} with \begin{equation}\label{eq:reg_assumption_bound} \max_{0\le t\le T} \left( \|u(\cdot,t)\|_{H^4(0,l)} + \|u_t(\cdot,t)\|_{H^4(0,l)} + \|u_{tt}(\cdot,t)\|_{H^3(0,l)} \right) \le C_u. \end{equation} In particular, \[ u_{xxxx t}\in C\bigl([0,T];L^2(0,l)\bigr), \qquad u_{xxtt}\in C\bigl([0,T];H^1(0,l)\bigr). \] Assumption \eqref{eq:reg_assumption_smooth} excludes the weak initial singularity that commonly occurs in time-fractional evolution problems. The treatment of such nonsmooth solutions requires time-dependent consistency estimates and a nonuniform discrete fractional Gr\"onwall argument and is beyond the scope of the present analysis.

\textbf{Consistency errors.}
Define the consistency errors by
\begin{align}
\boldsymbol\xi_\alpha^k
&:=
\delta_t^\alpha\mathbf U^k
-
R_h\partial_t^\alpha u(\cdot,t_k),
\label{eq:caputo_consistency_def}
\\
\boldsymbol\xi_x^k
&:=
L_h\mathbf U^k
-
R_hu_{xx}(\cdot,t_k),
\label{eq:space_consistency_def}
\\
\boldsymbol\xi_{xt}^k
&:=
\frac{L_h\mathbf U^k-L_h\mathbf U^{k-1}}{\tau_k}
-
R_hu_{xxt}(\cdot,t_k),
\label{eq:pseudo_consistency_def}
\end{align}
for $1\le k\le M$.

For every $v\in H^4(0,l)$, the standard central-difference estimate
gives
\begin{equation}\label{eq:central_difference_general}
\|L_hR_hv-R_hv_{xx}\|_h
\le Ch^2\|v\|_{H^4(0,l)},
\end{equation}
where $C$ is independent of $h$. Consequently,
\begin{equation}\label{eq:space_consistency_bound}
\|\boldsymbol\xi_x^k\|_h
\le Ch^2,
\qquad 1\le k\le M.
\end{equation}

To estimate the consistency error associated with the pseudo-parabolic
term, introduce
\[
\overline{u_t}^{\,k}
:=
\frac1{\tau_k}
\int_{t_{k-1}}^{t_k}u_t(\cdot,s)\,ds.
\]
Then
\[
\frac{\mathbf U^k-\mathbf U^{k-1}}{\tau_k}
=
R_h\overline{u_t}^{\,k},
\]
and hence
\begin{align}
\boldsymbol\xi_{xt}^k
&=
L_hR_h\overline{u_t}^{\,k}
-
R_hu_{xxt}(\cdot,t_k)
\notag\\
&=
\left(
L_hR_h\overline{u_t}^{\,k}
-
R_h(\overline{u_t}^{\,k})_{xx}
\right)
+
R_h\left(
(\overline{u_t}^{\,k})_{xx}
-u_{xxt}(\cdot,t_k)
\right).
\label{eq:pseudo_consistency_decomp}
\end{align}

By \eqref{eq:central_difference_general} and
\eqref{eq:reg_assumption_bound},
\begin{equation}\label{eq:pseudo_space_part}
\left\|
L_hR_h\overline{u_t}^{\,k}
-
R_h(\overline{u_t}^{\,k})_{xx}
\right\|_h
\le Ch^2.
\end{equation}

For the second term in \eqref{eq:pseudo_consistency_decomp}, set
\[
z^k
:=
(\overline{u_t}^{\,k})_{xx}
-u_{xxt}(\cdot,t_k).
\]
Then
\[
z^k
=
\frac1{\tau_k}
\int_{t_{k-1}}^{t_k}
\left(
u_{xxt}(\cdot,s)-u_{xxt}(\cdot,t_k)
\right)\,ds.
\]
Using the fundamental theorem of calculus, we obtain
\begin{align*}
\|z^k\|_{H^1(0,l)}
&\le
\frac1{\tau_k}
\int_{t_{k-1}}^{t_k}
\int_s^{t_k}
\|u_{xxtt}(\cdot,\eta)\|_{H^1(0,l)}
\,d\eta\,ds
\\
&\le C\tau_k.
\end{align*}
Moreover, the one-dimensional Sobolev embedding
$H^1(0,l)\hookrightarrow C[0,l]$ gives
\begin{align}
\|R_hz^k\|_h^2
&=
h\sum_{i=1}^{N-1}|z^k(x_i)|^2
\le
l\|z^k\|_{L^\infty(0,l)}^2
\le
C\|z^k\|_{H^1(0,l)}^2.
\label{eq:nodal_restriction_bound}
\end{align}
Therefore,
\[
\|R_hz^k\|_h\le C\tau_k.
\]
Combining this estimate with \eqref{eq:pseudo_space_part} yields
\begin{equation}\label{eq:pseudo_consistency_bound}
\|\boldsymbol\xi_{xt}^k\|_h
\le
C\bigl(h^2+\tau_k\bigr),
\qquad 1\le k\le M.
\end{equation}

Under the smoothness assumption \eqref{eq:reg_assumption_smooth}, the standard consistency estimate for the nonuniform $L1$ formula gives
\begin{equation}\label{eq:L1_smooth_consistency}
\max_{1\le k\le M}
\|\boldsymbol\xi_\alpha^k\|_h
\le
C\tau_{\max}^{\,2-\alpha}.
\end{equation}
For consistency and convergence analyses of the $L1$ approximation on graded temporal meshes, see, for example, \cite{Stynes17}. The constant $C$ in \eqref{eq:L1_smooth_consistency} is independent of $h$ and
$M$, but may depend on $T$, $\alpha$, the fixed grading parameter $r$, and the regularity constant $C_u$.

Estimate \eqref{eq:L1_smooth_consistency} is used only under \eqref{eq:reg_assumption_smooth}. In general, it does not hold uniformly over all time levels when
\[
u_{tt}(t)=O(t^{\alpha-2})
\qquad\text{as }t\to0^+.
\]

\textbf{Error equation.}
Evaluating the continuous problem at the spatial grid points and subtracting the fully discrete scheme gives
\begin{equation}\label{eq:error_eq_graded_full}
\delta_t^\alpha\mathbf e^k
-
L_h\mathbf e^k
-
\mu^k
\frac{L_h\mathbf e^k-L_h\mathbf e^{k-1}}{\tau_k}
=
\boldsymbol\rho^k,
\qquad
1\le k\le M,
\end{equation}
where $\mathbf e^0=\mathbf0$ and
\begin{equation}\label{eq:rho_definition}
\boldsymbol\rho^k
=
\boldsymbol\xi_\alpha^k
-
\boldsymbol\xi_x^k
-
\mu^k\boldsymbol\xi_{xt}^k.
\end{equation}
Using \eqref{eq:ass_mu},
\eqref{eq:space_consistency_bound},
\eqref{eq:pseudo_consistency_bound}, and
\eqref{eq:L1_smooth_consistency}, we obtain
\begin{equation}\label{eq:rho_bound_full}
\|\boldsymbol\rho^k\|_h
\le
C\left(
h^2+\tau_k+\tau_{\max}^{\,2-\alpha}
\right),
\qquad
1\le k\le M.
\end{equation}

\begin{theorem}[Convergence of the fully discrete scheme]
\label{thm:conv_graded_full}
Assume that \eqref{eq:ass_mu},
\eqref{eq:reg_assumption_smooth}, and
\eqref{eq:reg_assumption_bound} hold. Let
$\{\mathbf u^k\}_{k=0}^{M}$ be generated by
\eqref{eq:scheme_matrix_form}--\eqref{eq:history_term}, and let
$\mathbf e^k=\mathbf U^k-\mathbf u^k$. Then there exists a constant
$C>0$, independent of $h$ and $M$, such that
\begin{align}
&\max_{0\le k\le M}\|\mathbf e^k\|_h
+
\max_{0\le k\le M}
\sqrt{\mu^k}\,
\|\nabla_h\mathbf e^k\|_h
+
\left(
\sum_{k=1}^{M}
\tau_k\|\nabla_h\mathbf e^k\|_h^2
\right)^{1/2}
\notag\\
&\qquad\le
C\left(
h^2+\tau_{\max}
+\tau_{\max}^{\,2-\alpha}
\right).
\label{eq:conv_rate_full}
\end{align}
Since $2-\alpha>1$, the estimate simplifies, for sufficiently small
$\tau_{\max}$, to
\begin{align}
&\max_{0\le k\le M}\|\mathbf e^k\|_h
+
\max_{0\le k\le M}
\|\nabla_h\mathbf e^k\|_h
+
\left(
\sum_{k=1}^{M}
\tau_k\|\nabla_h\mathbf e^k\|_h^2
\right)^{1/2}
\notag\\
&\qquad\le
C\left(h^2+\tau_{\max}\right).
\label{eq:conv_rate_simplified}
\end{align}
For the graded mesh with fixed $r\ge1$, one has
$\tau_{\max}=O(M^{-1})$. Therefore, the method is second-order convergent in space in the discrete nodal norms appearing above and first-order convergent in time. The temporal order is limited by the backward-difference approximation of the pseudo-parabolic term
$u_{xxt}$.
\end{theorem}

\begin{proof}
Applying the stability estimate of
Theorem~\ref{thm:stab_graded} to the error equation
\eqref{eq:error_eq_graded_full}, and using
$\mathbf e^0=\mathbf0$, gives
\begin{align}
&\max_{0\le k\le M}\|\mathbf e^k\|_h^2
+
\max_{0\le k\le M}
\mu^k\|\nabla_h\mathbf e^k\|_h^2
+
\sum_{k=1}^{M}
\tau_k\|\nabla_h\mathbf e^k\|_h^2
\notag\\
&\qquad\le
C\sum_{k=1}^{M}
\tau_k\|\boldsymbol\rho^k\|_h^2.
\label{eq:error_stability_applied}
\end{align}

By \eqref{eq:rho_bound_full},
\[
\|\boldsymbol\rho^k\|_h^2
\le
C\left(
h^4+\tau_k^2+
\tau_{\max}^{\,2(2-\alpha)}
\right).
\]
Consequently,
\begin{align*}
\sum_{k=1}^{M}
\tau_k\|\boldsymbol\rho^k\|_h^2
&\le
C\left[
h^4\sum_{k=1}^{M}\tau_k
+
\sum_{k=1}^{M}\tau_k^3
+
\tau_{\max}^{\,2(2-\alpha)}
\sum_{k=1}^{M}\tau_k
\right]
\\
&\le
C\left(
h^4+\tau_{\max}^2
+\tau_{\max}^{\,2(2-\alpha)}
\right),
\end{align*}
where we have used
\[
\sum_{k=1}^{M}\tau_k=T
\]
and
\[
\sum_{k=1}^{M}\tau_k^3
\le
\tau_{\max}^2
\sum_{k=1}^{M}\tau_k
=
T\tau_{\max}^2.
\]
Taking square roots in \eqref{eq:error_stability_applied} proves
\eqref{eq:conv_rate_full}.

Since $0<\alpha<1$, one has $2-\alpha>1$, and therefore
\[
\tau_{\max}^{\,2-\alpha}
=o(\tau_{\max})
\qquad\text{as }\tau_{\max}\to0.
\]
Moreover, $\mu^k\ge\mu_0>0$, so the weighted discrete-gradient term
controls $\|\nabla_h\mathbf e^k\|_h$. Hence
\eqref{eq:conv_rate_simplified} follows.
\end{proof}

\section{Numerical reconstruction of the initial state}
\label{sec:inverse_num}

In this section, we describe the fully discrete reconstruction of the
unknown initial state $u_0=u(\cdot,0)$ from the final-time measurement
$\psi=u(\cdot,T)$ for problem
\eqref{1.1}--\eqref{finalT}. 

The reconstruction is based on a discretise-then-regularise strategy.
First, the discrete forward operator is constructed by repeated applications of the graded-mesh $L1$ finite-difference solver developed in Section~\ref{sec:direct_num}. The resulting finite-dimensional inverse system is then stabilised by zeroth-order Tikhonov regularisation.

\subsection{Discrete forward operator and inverse relation}
\label{subsec:disc_inverse}
We use the spatial and temporal grids introduced in
Subsection~\ref{subsec:disc}. For an initial vector
$\mathbf u_{0,h}\in\mathbb R^{N-1}$, the fully discrete scheme produces the trajectory $\{\mathbf u^k\}_{k=0}^{M}$, with $\mathbf u^0=\mathbf u_{0,h}$.

Let
\[
\boldsymbol\psi_h
:=
R_h\psi
=
\bigl(\psi(x_1),\dots,\psi(x_{N-1})\bigr)^\top
\]
denote the exact nodal final-time data. When the measurements are perturbed, we write
\[
\boldsymbol\psi_h^\delta
=
\boldsymbol\psi_h+\boldsymbol\eta_h^\delta,
\]
where $\boldsymbol\eta_h^\delta$ denotes the measurement error and $\delta\ge0$ represents its level. The noise-free case corresponds to $\delta=0$.

Let
\[
\mathcal S_{h,\tau}:
\mathbb R^{N-1}\times\mathcal F_h
\longrightarrow
\mathbb R^{N-1}
\]
denote the discrete terminal-state map induced by the fully discrete scheme \eqref{eq:scheme_matrix_form}--\eqref{eq:history_term}. Thus, for a prescribed initial vector $\mathbf v$ and a discrete source $f$,
\[
\mathcal S_{h,\tau}(\mathbf v;f)
=
\mathbf u^M(\mathbf v;f).
\]
Here and below, the dependence of the discrete operators on the temporal mesh, the fractional order, and the coefficient $\mu$ is suppressed when no confusion can arise.

Since the fully discrete scheme is linear with respect to both the solution and the source, the terminal vector admits the decomposition
\begin{equation}\label{eq:affine_split}
\mathbf u^M(\mathbf u_{0,h};f)
=
\mathbf u^M(\mathbf u_{0,h};0)
+
\mathbf u^M(\mathbf0;f).
\end{equation}
Define the forced terminal contribution by
\begin{equation}\label{eq:forced_terminal}
\mathbf g_h
:=
\mathbf u^M(\mathbf0;f),
\end{equation}
and define the homogeneous discrete forward operator
\[
\mathbf F_h\in\mathbb R^{(N-1)\times(N-1)}
\]
through
\begin{equation}\label{eq:homogeneous_forward}
\mathbf u^M(\mathbf v;0)
=
\mathbf F_h\mathbf v,
\qquad
\mathbf v\in\mathbb R^{N-1}.
\end{equation}
Consequently,
\begin{equation}\label{eq:discrete_affine_map}
\mathbf u^M(\mathbf u_{0,h};f)
=
\mathbf F_h\mathbf u_{0,h}+\mathbf g_h.
\end{equation}

For exact data, the discrete inverse relation is
\begin{equation}\label{eq:lin_system_inv_exact}
\mathbf F_h\mathbf u_{0,h}
=
\boldsymbol\psi_h-\mathbf g_h.
\end{equation}
For perturbed measurements, we define
\begin{equation}\label{eq:discrete_data_vector}
\mathbf d_h^\delta
:=
\boldsymbol\psi_h^\delta-\mathbf g_h
\end{equation}
and seek an approximation to the solution of
\begin{equation}\label{eq:lin_system_inv}
\mathbf F_h\mathbf u_{0,h}
\approx
\mathbf d_h^\delta.
\end{equation}

\paragraph{Construction of the discrete forward operator.}
Let
\[
\{\mathbf e^{(m)}\}_{m=1}^{N-1}
\]
be the canonical basis of $\mathbb R^{N-1}$. For each
$m=1,\dots,N-1$, we solve the homogeneous fully discrete problem with
\[
\mathbf u^0=\mathbf e^{(m)},
\qquad
f\equiv0,
\]
and record the corresponding terminal vector
\[
\mathbf k^{(m)}
:=
\mathbf u^M(\mathbf e^{(m)};0).
\]
By linearity, $\mathbf k^{(m)}$ is the $m$-th column of
$\mathbf F_h$, and therefore
\begin{equation}\label{eq:Fh_assembly}
\mathbf F_h
=
\bigl[
\mathbf k^{(1)}
\ \mathbf k^{(2)}
\ \cdots
\ \mathbf k^{(N-1)}
\bigr].
\end{equation}

This construction requires $N-1$ homogeneous forward solves and is computationally practical for the moderate spatial dimensions used in the numerical experiments. For larger systems, explicit assembly may be replaced by a matrix-free iterative method based on products with
$\mathbf F_h$ and its transpose. Such an implementation generally requires the corresponding discrete adjoint operator.

\subsection{Tikhonov regularisation}
\label{subsec:tikhonov}
The continuous inverse problem is uniquely solvable and satisfies the stability estimate established in Section~\ref{sec:wp}. Nevertheless, the discrete reconstruction may be affected by measurement noise,
discretisation errors, and finite-precision arithmetic. We therefore use Tikhonov regularisation as a numerical stabilisation procedure.

For grid vectors, let
\[
(\mathbf v,\mathbf w)_h
:=
h\sum_{i=1}^{N-1}v_iw_i,
\qquad
\|\mathbf v\|_h
:=
(\mathbf v,\mathbf v)_h^{1/2}.
\]
The zeroth-order Tikhonov approximation is defined as the unique minimiser of
\begin{equation}\label{eq:tikhonov_fun}
\mathcal J_\lambda(\mathbf v)
=
\frac12
\|\mathbf F_h\mathbf v-\mathbf d_h^\delta\|_h^2
+
\frac{\lambda}{2}\|\mathbf v\|_h^2,
\qquad
\lambda>0.
\end{equation}
Since both terms contain the same factor $h$, the first-order optimality condition is equivalent to
\begin{equation}\label{eq:normal_eq}
\left(
\mathbf F_h^\top\mathbf F_h+\lambda I
\right)
\mathbf u_{0,h}^{\lambda,\delta}
=
\mathbf F_h^\top\mathbf d_h^\delta.
\end{equation}
The coefficient matrix in \eqref{eq:normal_eq} is symmetric positive definite for every $\lambda>0$; hence, the minimiser is unique.

For the moderate dimensions considered in the numerical experiments, \eqref{eq:normal_eq} can be solved by a dense Cholesky solver. From the viewpoint of numerical linear algebra, an equivalent augmented least-squares formulation,
\[
\begin{bmatrix}
\mathbf F_h\\
\sqrt{\lambda}\,I
\end{bmatrix}
\mathbf u_{0,h}^{\lambda,\delta}
\approx
\begin{bmatrix}
\mathbf d_h^\delta\\
\mathbf0
\end{bmatrix},
\]
may instead be solved by QR factorisation or the singular value decomposition. This avoids the explicit squaring of the condition number associated with the formation of $\mathbf F_h^\top\mathbf F_h$.

\subsection{Choice of the regularisation parameter}
\label{subsec:parameter_choice}
The regularisation parameter $\lambda$ determines the balance between
fidelity to the measured terminal data and suppression of unstable or
noise-dominated components. A value of $\lambda$ that is too small may
lead to excessive sensitivity to measurement and discretisation errors,
whereas a value that is too large introduces an unnecessarily strong
regularisation bias.

In the numerical experiments, $\lambda$ is selected using the
L-curve criterion. Let
\[
\Lambda
=
\{\lambda_1,\dots,\lambda_J\}
\]
be a logarithmically distributed set of positive candidate parameters.
For each $\lambda\in\Lambda$, we compute
$\mathbf u_{0,h}^{\lambda,\delta}$ from
\eqref{eq:normal_eq} and evaluate the residual and solution norms
\begin{align}
\rho(\lambda)
&:=
\left\|
\mathbf F_h\mathbf u_{0,h}^{\lambda,\delta}
-\mathbf d_h^\delta
\right\|_h,
\label{eq:lcurve_residual}
\\
\eta(\lambda)
&:=
\left\|
\mathbf u_{0,h}^{\lambda,\delta}
\right\|_h.
\label{eq:lcurve_solution}
\end{align}
The L-curve is the parametric curve
\begin{equation}\label{eq:lcurve}
\lambda
\longmapsto
\bigl(
\log\rho(\lambda),
\log\eta(\lambda)
\bigr).
\end{equation}
The selected parameter is taken near the corner of this curve, where the transition from the data-fitting regime to the regularisation-dominated regime occurs. In the discrete implementation, the corner is identified by locating a candidate parameter near the maximum curvature of the log--log L-curve.

For the numerical examples considered below, representative L-curves indicated a parameter of order $10^{-10}$ for the noise-free tests and of order $10^{-6}$ for the noisy-data tests. Accordingly, we use
$
\lambda=10^{-10}
$
in the noise-free validation experiments and
$
\lambda=10^{-6}
$
in the experiments with perturbed final-time data. These are empirical choices for the present test problems and are not intended as universal regularisation parameters. The values are kept fixed within each group of experiments in order to facilitate comparisons between different fractional orders and test configurations.

\subsection{Reconstruction procedure}
\label{subsec:alg_inverse}
Algorithm~\ref{alg:inverse_frac_pp} summarises the reconstruction procedure.

\begin{algorithm}[H]
\caption{Tikhonov reconstruction using the time-fractional
pseudo-parabolic forward solver}
\label{alg:inverse_frac_pp}
\begin{algorithmic}[1]
\Require Spatial grid size $N$; temporal mesh
$\{t_k\}_{k=0}^{M}$; measured terminal data
$\boldsymbol\psi_h^\delta\in\mathbb R^{N-1}$;
regularisation parameter $\lambda>0$; coefficient $\mu(t)$;
source $f(x,t)$; and the fully discrete forward solver
\eqref{eq:scheme_matrix_form}--\eqref{eq:history_term}.
\Ensure Reconstructed initial vector
$\mathbf u_{0,h}^{\lambda,\delta}\in\mathbb R^{N-1}$.

\State Solve the forward problem with
$\mathbf u^0=\mathbf0$ and the prescribed source $f$.
\State Set
$\mathbf g_h:=\mathbf u^M(\mathbf0;f)$.

\For{$m=1,\dots,N-1$}
    \State Solve the homogeneous forward problem with
    $\mathbf u^0=\mathbf e^{(m)}$ and $f\equiv0$.
    \State Set
    $\mathbf F_h(:,m):=\mathbf u^M(\mathbf e^{(m)};0)$.
\EndFor

\State Form
$\mathbf d_h^\delta:=\boldsymbol\psi_h^\delta-\mathbf g_h$.
\State Compute $\mathbf u_{0,h}^{\lambda,\delta}$ from
\eqref{eq:normal_eq}.
\State Return $\mathbf u_{0,h}^{\lambda,\delta}$.
\end{algorithmic}
\end{algorithm}

\paragraph{Verification and error measures.}
To verify the reconstruction, we solve the forward problem once more using the reconstructed initial state and define the corresponding terminal state by
\[
\widehat{\boldsymbol\psi}_h^{\lambda,\delta}
:=
\mathbf u^M
\bigl(
\mathbf u_{0,h}^{\lambda,\delta};f
\bigr)
=
\mathbf F_h\mathbf u_{0,h}^{\lambda,\delta}
+\mathbf g_h.
\]

When the exact initial and terminal states are available, we set
\[
\mathbf u_{0,h}^{\mathrm{ex}}
:=
R_hu_0,
\qquad
\boldsymbol\psi_h^{\mathrm{ex}}
:=
R_h\psi.
\]
The reconstruction errors for the initial state are measured by
\begin{align}
E_{u_0,\infty}
&:=
\left\|
\mathbf u_{0,h}^{\lambda,\delta}
-
\mathbf u_{0,h}^{\mathrm{ex}}
\right\|_\infty
=
\max_{1\le j\le N-1}
\left|
u_{0,j}^{\lambda,\delta}
-u_0(x_j)
\right|,
\label{eq:error_u0_inf}
\\
E_{u_0,2}
&:=
\left\|
\mathbf u_{0,h}^{\lambda,\delta}
-
\mathbf u_{0,h}^{\mathrm{ex}}
\right\|_{2,h}
\notag\\
&=
\left(
h\sum_{j=1}^{N-1}
\left|
u_{0,j}^{\lambda,\delta}
-u_0(x_j)
\right|^2
\right)^{1/2}.
\label{eq:error_u0_l2}
\end{align}

The corresponding errors in the reconstructed terminal state are defined by
\begin{align}
E_{\psi,\infty}
&:=
\left\|
\widehat{\boldsymbol\psi}_h^{\lambda,\delta}
-
\boldsymbol\psi_h^{\mathrm{ex}}
\right\|_\infty
=
\max_{1\le j\le N-1}
\left|
\widehat{\psi}_j^{\lambda,\delta}
-\psi(x_j)
\right|,
\label{eq:error_psi_inf}
\\
E_{\psi,2}
&:=
\left\|
\widehat{\boldsymbol\psi}_h^{\lambda,\delta}
-
\boldsymbol\psi_h^{\mathrm{ex}}
\right\|_{2,h}
\notag\\
&=
\left(
h\sum_{j=1}^{N-1}
\left|
\widehat{\psi}_j^{\lambda,\delta}
-\psi(x_j)
\right|^2
\right)^{1/2}.
\label{eq:error_psi_l2}
\end{align}
Here, $\|\cdot\|_{2,h}$ denotes the discrete $L^2(0,l)$-norm.

\section{Numerical validation and experiments}\label{sec:numexp}

In this section, we present numerical experiments that demonstrate the practical behaviour of the proposed
regularised reconstruction of the \emph{initial state} $u_0(x)=u(x,0)$ from the \emph{final-time measurement}
$\psi(x)=u(x,T)$.
All computations employ the direct solver from Section~\ref{sec:direct_num} and the discrete inverse formulation
from Section~\ref{sec:inverse_num}.

\subsection{Test setting and implementation details}\label{subsec:test}
Let $0<\alpha<1$ and $(x,t)\in[0,l]\times[0,T]$. We consider the manufactured solution
\[
\mu(t)=1+t,
\qquad
u(x,t)=\bigl(1+t^{\alpha+1}\bigr)\sin\!\Bigl(\frac{\pi x}{l}\Bigr),
\]
together with the forcing term chosen so that $u$ satisfies the time-fractional pseudo-parabolic model,
namely
\[
f(x,t)=\Bigl[
\Gamma(\alpha+2)\,t
+\Bigl(\frac{\pi}{l}\Bigr)^2\bigl(1+t^{\alpha+1}\bigr)
+\mu(t)\Bigl(\frac{\pi}{l}\Bigr)^2(\alpha+1)t^\alpha
\Bigr]
\sin\!\Bigl(\frac{\pi x}{l}\Bigr).
\]
Accordingly, the final-time measurement and initial state are
\[
\psi(x)=u(x,T)
=\bigl(1+T^{\alpha+1}\bigr)\sin\!\Bigl(\frac{\pi x}{l}\Bigr),
\qquad
u_0(x)=u(x,0)
=\sin\!\Bigl(\frac{\pi x}{l}\Bigr).
\]

The final-time measurement is sampled at the interior spatial nodes, as described in
Section~\ref{sec:inverse_num}, and the homogeneous Dirichlet boundary conditions are imposed at all time levels.

\subsection{Noise-free validation}\label{subsec:accuracy}
We begin with a noise-free study to verify that the direct solver and the reconstruction procedure are
mutually consistent. In this setting, the final-time measurement is sampled from the exact formula, i.e.,
\[
\psi_i=\psi(x_i),
\qquad i=1,\dots,N-1.
\]

We fix $T=1$ and refine the space--time grids according to
\[
(N,M)\in\{(50,50),(100,100),(200,200),(400,400)\}.
\]
The reconstruction is performed using the procedure described in
Section~\ref{sec:inverse_num}. In accordance with the L-curve parameter-choice procedure described in
Subsection~\ref{subsec:parameter_choice}, we set
$
\lambda=10^{-10}
$
in the noise-free experiments.

The maximum-norm and discrete $L^2$ errors for $u_0$ and the corresponding final-time errors for
$\psi$ are listed in Table~\ref{tab1}. The decay of the initial-state errors under refinement confirms the
convergence of the overall discrete reconstruction procedure. The final-time errors remain close to the
regularisation level, demonstrating consistency between the inverse reconstruction and the direct solver.

\begin{table}[htbp]
\centering
\caption{Reconstruction errors in the $L^\infty$- and $L^2$-norms with noise-free terminal data.}
\begin{tabular}{c c c c c c}
\hline
$\alpha$ & $(N,M)$ & $E_{u_0,\infty}$ & $E_{u_0,2}$ & $E_{\psi,\infty}$ & $E_{\psi,2}$\\
\hline
\multirow{4}{*}{0.1}
& $(50,50)$   & 5.477e-03 & 3.873e-03 & 1.906e-10 & 1.348e-10 \\
& $(100,100)$ & 2.545e-03 & 1.799e-03 & 1.923e-10 & 1.360e-10 \\
& $(200,200)$ & 1.224e-03 & 8.653e-04 & 1.931e-10 & 1.365e-10 \\
& $(400,400)$ & 5.997e-04 & 4.241e-04 & 1.937e-10 & 1.371e-10 \\
\hline
\multirow{4}{*}{0.3}
& $(50,50)$   & 1.371e-02 & 9.693e-03 & 1.882e-10 & 1.331e-10 \\
& $(100,100)$ & 6.698e-03 & 4.736e-03 & 1.908e-10 & 1.349e-10 \\
& $(200,200)$ & 3.310e-03 & 2.340e-03 & 1.918e-10 & 1.357e-10 \\
& $(400,400)$ & 1.645e-03 & 1.163e-03 & 1.870e-10 & 1.323e-10 \\
\hline
\multirow{4}{*}{0.5}
& $(50,50)$   & 2.138e-02 & 1.512e-02 & 1.861e-10 & 1.316e-10 \\
& $(100,100)$ & 1.055e-02 & 7.460e-03 & 1.890e-10 & 1.336e-10 \\
& $(200,200)$ & 5.239e-03 & 3.704e-03 & 1.886e-10 & 1.334e-10 \\
& $(400,400)$ & 2.610e-03 & 1.845e-03 & 2.028e-10 & 1.434e-10 \\
\hline
\multirow{4}{*}{0.7}
& $(50,50)$   & 2.893e-02 & 2.046e-02 & 1.840e-10 & 1.301e-10 \\
& $(100,100)$ & 1.433e-02 & 1.013e-02 & 1.879e-10 & 1.328e-10 \\
& $(200,200)$ & 7.124e-03 & 5.038e-03 & 1.889e-10 & 1.336e-10 \\
& $(400,400)$ & 3.550e-03 & 2.510e-03 & 1.941e-10 & 1.370e-10 \\
\hline
\multirow{4}{*}{0.9}
& $(50,50)$   & 3.695e-02 & 2.613e-02 & 1.821e-10 & 1.287e-10 \\
& $(100,100)$ & 1.835e-02 & 1.297e-02 & 1.864e-10 & 1.318e-10 \\
& $(200,200)$ & 9.133e-03 & 6.458e-03 & 1.891e-10 & 1.336e-10 \\
& $(400,400)$ & 4.552e-03 & 3.218e-03 & 1.910e-10 & 1.349e-10 \\
\hline
\end{tabular}
\label{tab1}
\end{table}

Figure~\ref{fig1} shows representative reconstructions for several values of $\alpha$ on the fixed grid
$N=M=50$. For each $\alpha$, the left panel compares the exact initial state $u_0(x)$ with its regularised
reconstruction $\widehat{u_0}(x)$, while the right panel compares the exact final-time measurement
$\psi(x)$ with the final-time state $u(x,T;\widehat{u_0})$ obtained by forward propagation of the
reconstructed initial state. The close agreement in both panels confirms the consistency of the inversion
procedure and the direct solver.

\begin{figure}[h]
  \centering
  \begin{minipage}{\linewidth}
    \centering
    \includegraphics[width=0.6\linewidth]{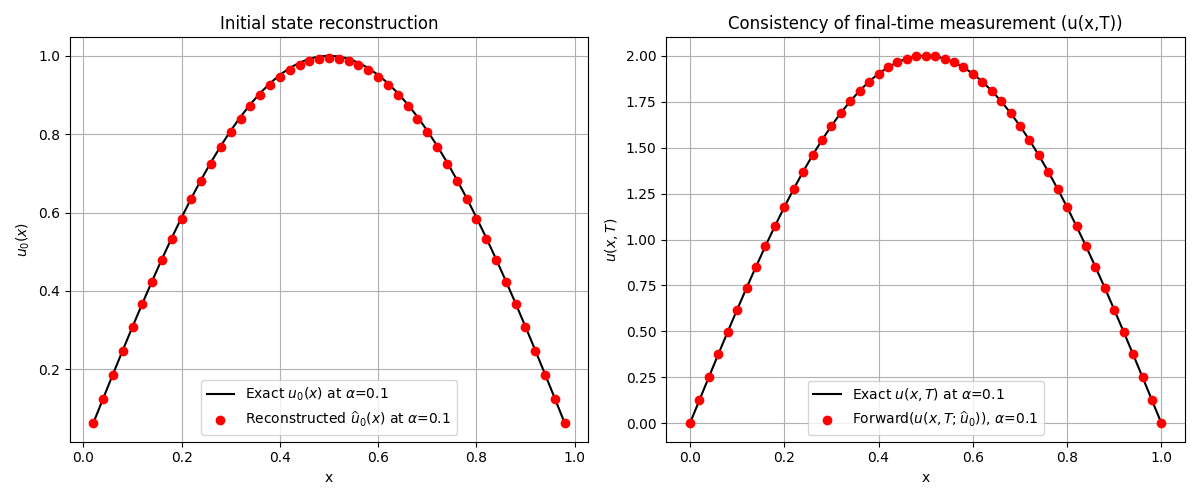}
  \end{minipage}\quad
  \begin{minipage}{\linewidth}
    \centering
    \includegraphics[width=0.6\linewidth]{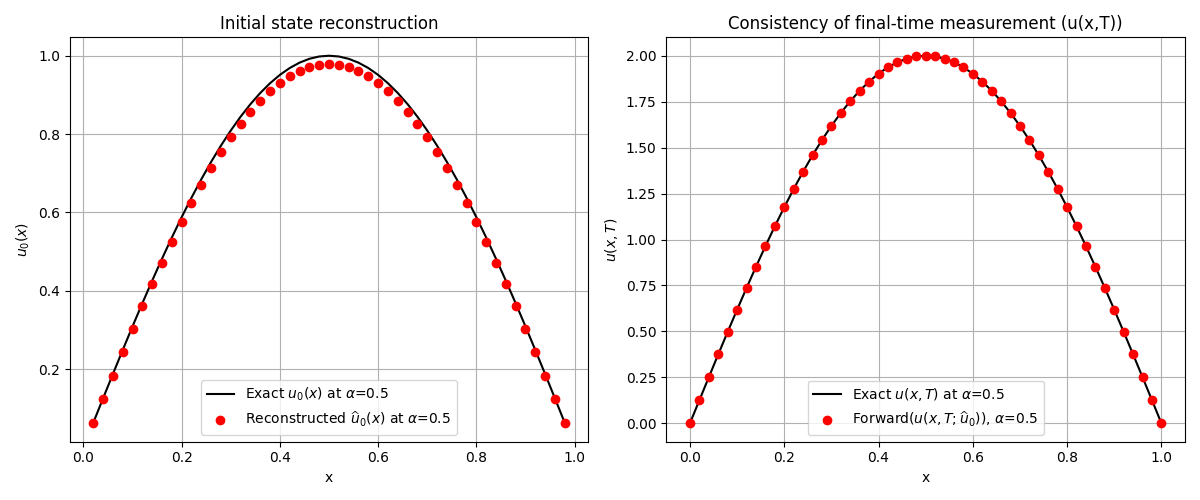}
  \end{minipage}\quad
  \begin{minipage}{\linewidth}
    \centering
    \includegraphics[width=0.6\linewidth]{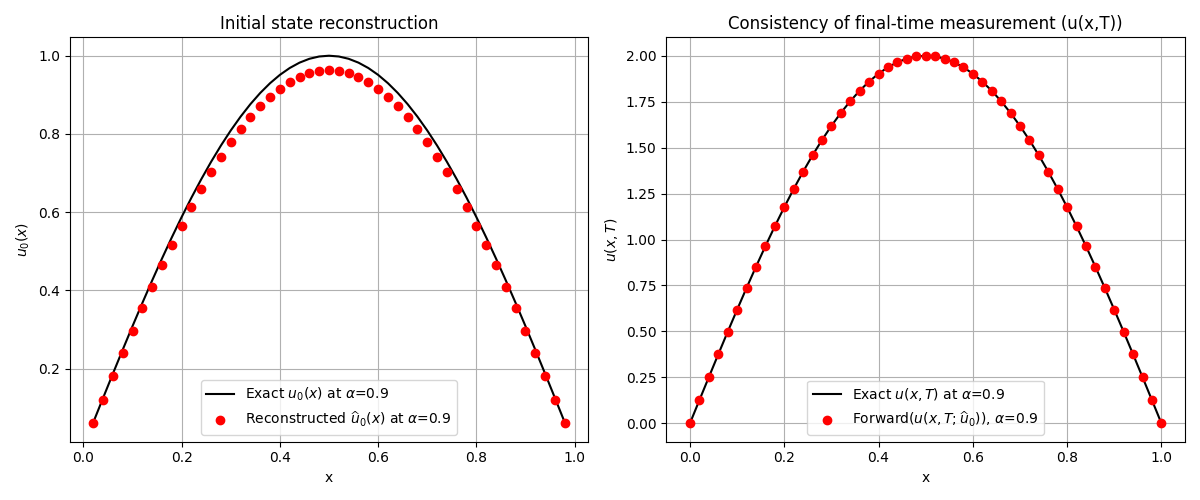}
  \end{minipage}
  \caption{Reconstruction results for different fractional orders.
  Left: exact initial state $u_0(x)$ and reconstructed state $\widehat{u_0}(x)$.
  Right: exact final-time measurement $\psi(x)$ and the final-time state
  $u(x,T;\widehat{u_0})$ generated by evolving the reconstructed initial state
  with the direct solver.}
  \label{fig1}
\end{figure}

Figure~\ref{fig2} compares the surface plot of the exact solution $u(x,t)$ with the reconstructed numerical
approximation obtained from the final-time measurement over the full time interval for several values of
$\alpha$. It also displays the corresponding absolute error surfaces. All simulations use the fixed grid
$N=M=200$.

\begin{figure}[h]
  \centering
  \begin{minipage}{\linewidth}
    \centering
    \includegraphics[width=0.6\linewidth]{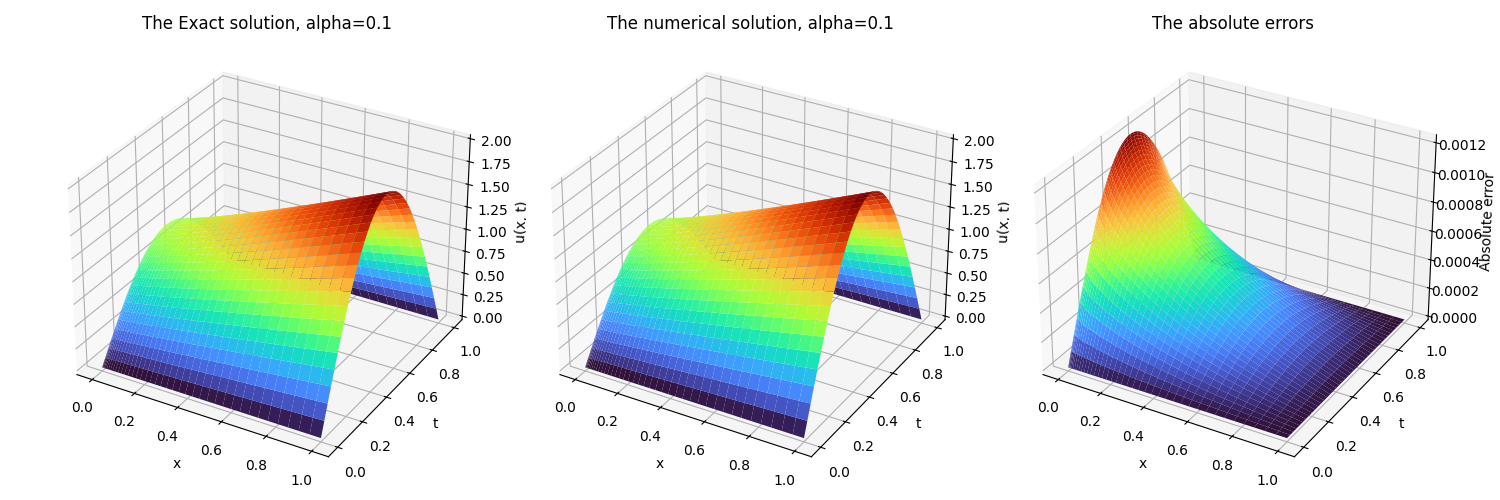}
  \end{minipage}\quad
  \begin{minipage}{\linewidth}
    \centering
    \includegraphics[width=0.6\linewidth]{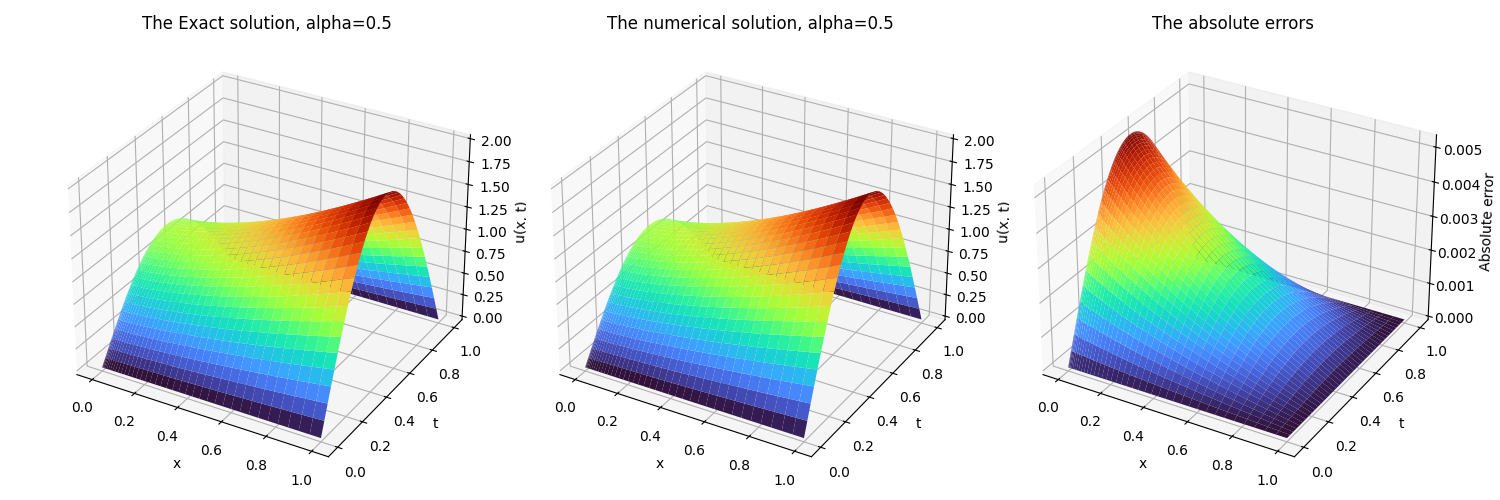}
  \end{minipage}\quad
  \begin{minipage}{\linewidth}
    \centering
    \includegraphics[width=0.6\linewidth]{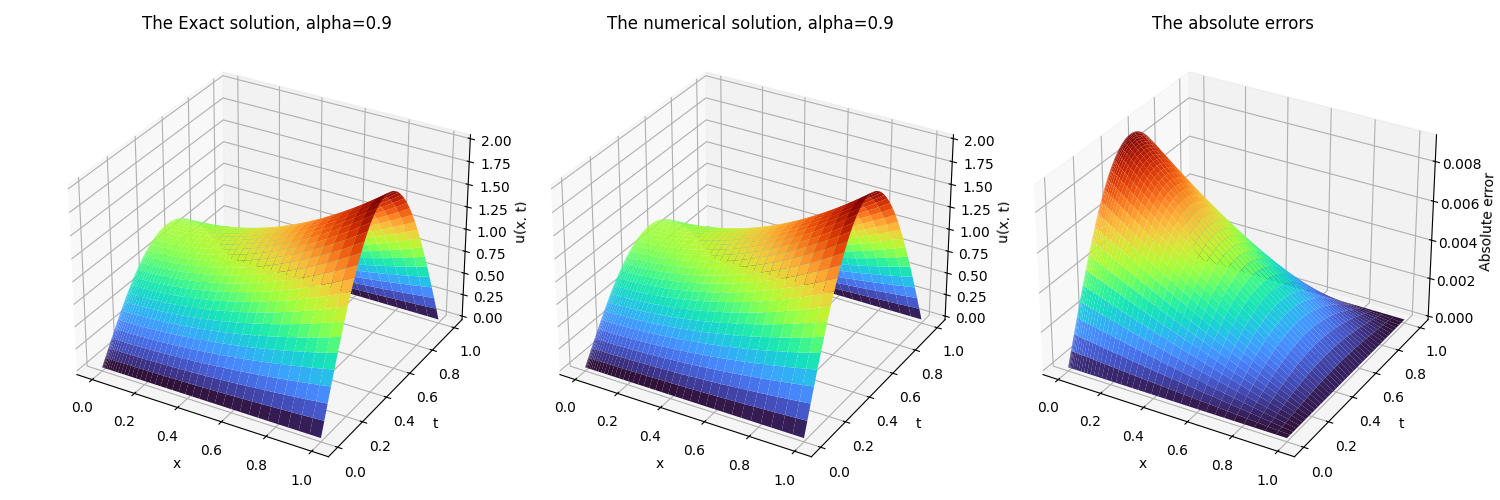}
  \end{minipage}
  \caption{Exact solution (left), reconstructed solution (centre), and absolute
  error (right) for $u(x,t)$ for several values of $\alpha$.}
  \label{fig2}
\end{figure}

\subsection{Reconstruction from noisy final-time data}\label{subsec:noise}
In practice, the final-time measurement $\psi(x)$ is corrupted by measurement noise. To assess robustness, we perturb the discrete final-time vector by additive Gaussian noise with a prescribed relative level.

Let $\boldsymbol{\psi}_h\in\mathbb{R}^{N-1}$ be the clean interior measurement. For $\delta>0$, we set
\begin{equation}\label{eq:noise_model_alt}
\boldsymbol{\psi}_h^{\delta}
=
\boldsymbol{\psi}_h
+
\sigma_\delta\,\boldsymbol{\xi},
\qquad
\boldsymbol{\xi}\sim\mathcal{N}(\mathbf{0},I_{N-1}),
\qquad
\sigma_\delta
:=
\delta\,\frac{\|\boldsymbol{\psi}_h\|_2}{\sqrt{N-1}}.
\end{equation}
With this choice,
\[
\frac{
\mathbb{E}\|\boldsymbol{\psi}_h^{\delta}-\boldsymbol{\psi}_h\|_2^2
}{
\|\boldsymbol{\psi}_h\|_2^2
}
=
\delta^2,
\]
so that $\delta$ represents the root-mean-square relative noise level.

For each noisy data vector $\boldsymbol{\psi}_h^\delta$, the initial state is reconstructed using the
procedure in Section~\ref{sec:inverse_num}. We use the fixed grid
$ N=M=100 $ and regularisation parameter
$
\lambda=10^{-6}.
$
The errors $E_{u_0,\infty}^{\delta}$, $E_{u_0,2}^{\delta}$,
$E_{\psi,\infty}^{\delta}$, and $E_{\psi,2}^{\delta}$ are computed using
the definitions given in Section~\ref{sec:inverse_num}.

Representative results for
$
\delta\in\{1\%,3\%,5\%\}
$
are displayed in Table~\ref{tab2}.

\begin{table}[htbp]
\centering
\caption{Influence of additive noise in the final-time measurement $\psi$ on the recovery of the initial
state $u_0$. Errors are reported for the fixed grid $N=M=100$ and regularisation parameter
$\lambda=10^{-6}$.}
\begin{tabular}{c c c c c c}
\hline
$\alpha$ & noise level $\delta$ & $E_{u_0,\infty}^{\delta}$ &
$E_{u_0,2}^{\delta}$ & $E_{\psi,\infty}^{\delta}$ & $E_{\psi,2}^{\delta}$\\
\hline
\multirow{3}{*}{0.1}
& $1\%$ & 7.536e-02 & 3.655e-02 & 3.724e-02 & 1.828e-02 \\
& $3\%$ & 2.231e-01 & 1.090e-01 & 1.117e-01 & 5.483e-02 \\
& $5\%$ & 3.708e-01 & 1.815e-01 & 1.862e-01 & 9.138e-02 \\
\hline
\multirow{3}{*}{0.3}
& $1\%$ & 7.776e-02 & 3.727e-02 & 3.724e-02 & 1.828e-02 \\
& $3\%$ & 2.255e-01 & 1.095e-01 & 1.117e-01 & 5.483e-02 \\
& $5\%$ & 3.732e-01 & 1.819e-01 & 1.862e-01 & 9.138e-02 \\
\hline
\multirow{3}{*}{0.5}
& $1\%$ & 7.998e-02 & 3.818e-02 & 3.724e-02 & 1.828e-02 \\
& $3\%$ & 2.277e-01 & 1.100e-01 & 1.117e-01 & 5.483e-02 \\
& $5\%$ & 3.753e-01 & 1.823e-01 & 1.862e-01 & 9.138e-02 \\
\hline
\multirow{3}{*}{0.7}
& $1\%$ & 8.226e-02 & 3.936e-02 & 3.724e-02 & 1.828e-02 \\
& $3\%$ & 2.299e-01 & 1.106e-01 & 1.117e-01 & 5.483e-02 \\
& $5\%$ & 3.776e-01 & 1.828e-01 & 1.862e-01 & 9.138e-02 \\
\hline
\multirow{3}{*}{0.9}
& $1\%$ & 8.518e-02 & 4.120e-02 & 3.724e-02 & 1.828e-02 \\
& $3\%$ & 2.328e-01 & 1.115e-01 & 1.117e-01 & 5.483e-02 \\
& $5\%$ & 3.805e-01 & 1.835e-01 & 1.862e-01 & 9.138e-02 \\
\hline
\end{tabular}
\label{tab2}
\end{table}

Table~\ref{tab2} quantifies the sensitivity of the reconstruction to perturbations in the final-time
measurement $\psi$. As expected for a backward problem, the reconstruction error increases as the noise
level $\delta$ grows. Nevertheless, the growth is controlled by the Tikhonov stabilisation: for
$\delta=1\%$, the recovered initial state remains accurate, with
$
E_{u_0,2}^{\delta}
\approx
(3.7\text{--}4.1)\times10^{-2},
\quad
E_{u_0,\infty}^{\delta}
\approx
(7.5\text{--}8.5)\times10^{-2}
$
across all tested fractional orders. When the noise is increased to
$\delta=3\%$ and $\delta=5\%$, the initial-state errors rise in a
near-proportional manner. For example, $E_{u_0,2}^{\delta}$ increases from
approximately $3.7\times10^{-2}$ to $1.1\times10^{-1}$ and then to
$1.8\times10^{-1}$. This indicates stable dependence on the data
perturbation rather than uncontrolled amplification.

A further observation is that the reported errors vary only mildly with $\alpha$ for a fixed value of
$\delta$. This suggests that, for the present test, the regularised inversion is not overly sensitive to
the fractional order and that the dominant factor affecting the accuracy is the measurement noise level.
In addition, the terminal-state errors $E_{\psi,\infty}^{\delta}$ and
$E_{\psi,2}^{\delta}$ remain comparable to the noise magnitude, showing
that the reconstructed initial state produces forward solutions that remain close to the exact terminal
state.

\begin{figure}[t]
\centering
\includegraphics[width=0.6\textwidth]{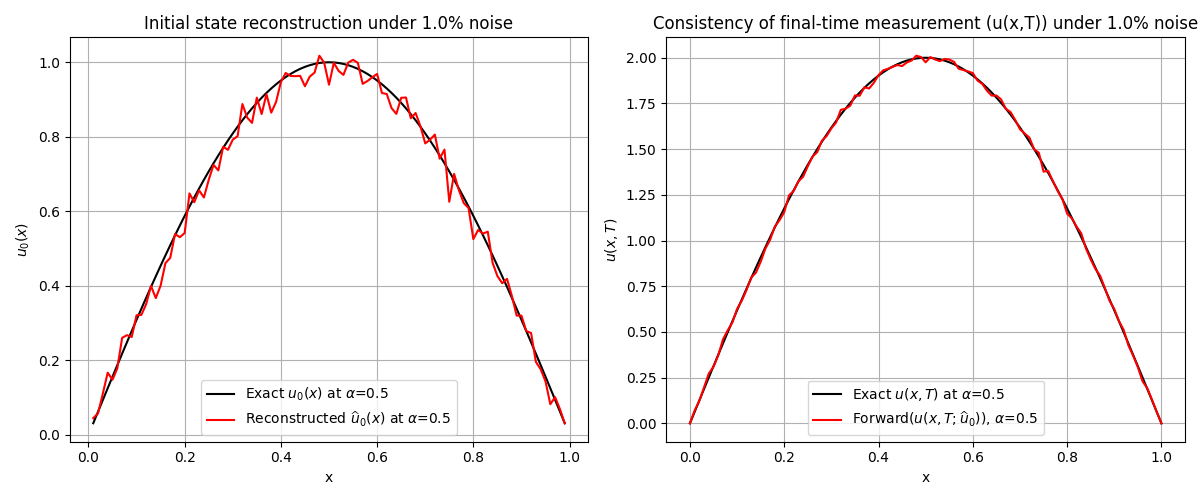}
\hfill
\includegraphics[width=0.6\textwidth]{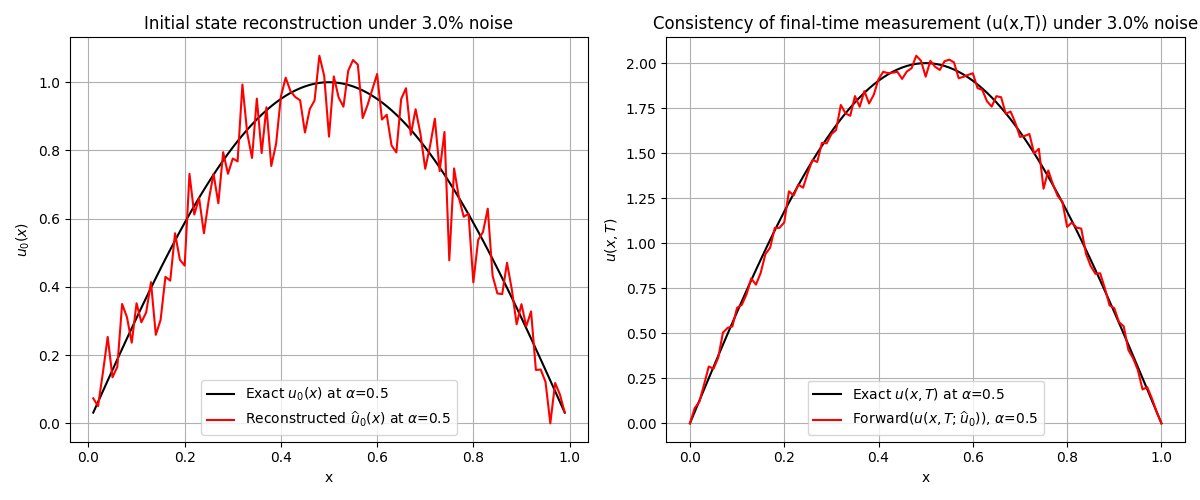}
\hfill
\includegraphics[width=0.6\textwidth]{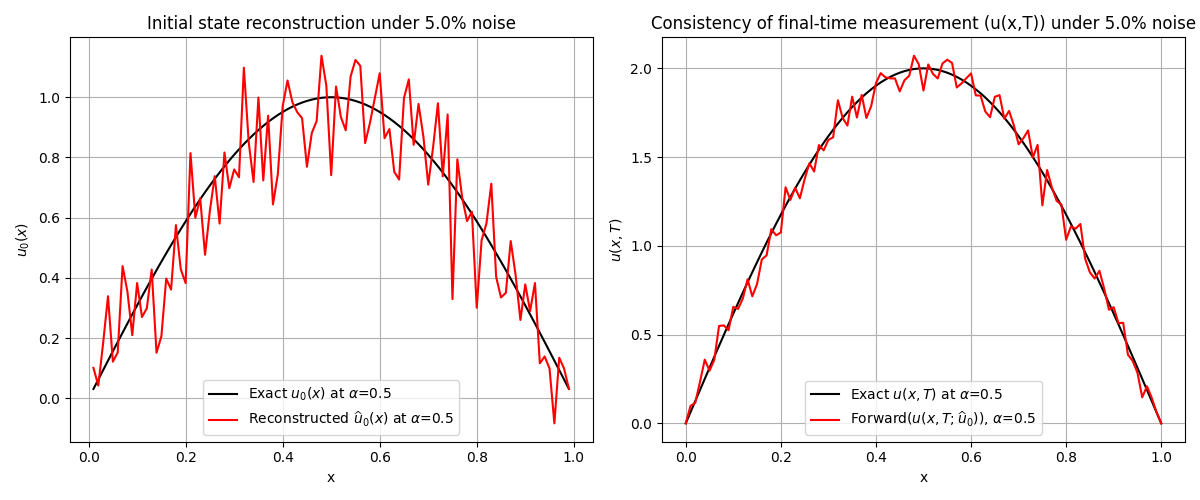}
\caption{Reconstructed initial states for noise levels $delta=1\%$ (top), $\delta=3\%$ (middle), and $\delta=5\%$ (bottom), compared with the exact initial state.} \label{fig:noise_pp}
\end{figure}
Figure~\ref{fig:noise_pp} illustrates the influence of measurement noise on the reconstructed initial state. For the noise level $\delta=1\%$, the reconstructed profile remains very close to the exact initial state. As the
noise level increases to $3\%$ and $5\%$, the deviation becomes more visible, particularly in the amplitude of the reconstructed solution.
Nevertheless, the main shape and spatial structure of $u_0(x)$ are preserved in all three cases, and no significant spurious oscillations are observed. This behaviour confirms the stabilising effect of the Tikhonov
regularisation and is consistent with the quantitative errors reported in Table~\ref{tab2}

Figure~\ref{fig3} shows surface plots of the reconstructed numerical solution $u(x,t)$, obtained from the
final-time measurement, for several noise levels over the full time interval with $\alpha=0.5$.

\begin{figure}[h]
  \centering
  \begin{minipage}{0.3\linewidth}
    \centering
    \includegraphics[width=\linewidth]{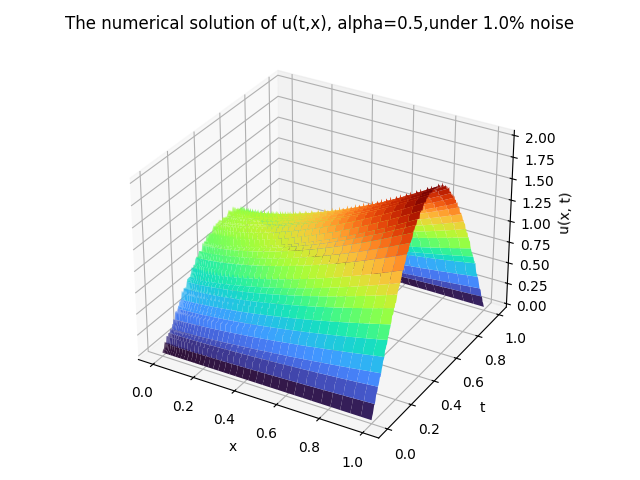}
  \end{minipage}
  \hfill
  \begin{minipage}{0.3\linewidth}
    \centering
    \includegraphics[width=\linewidth]{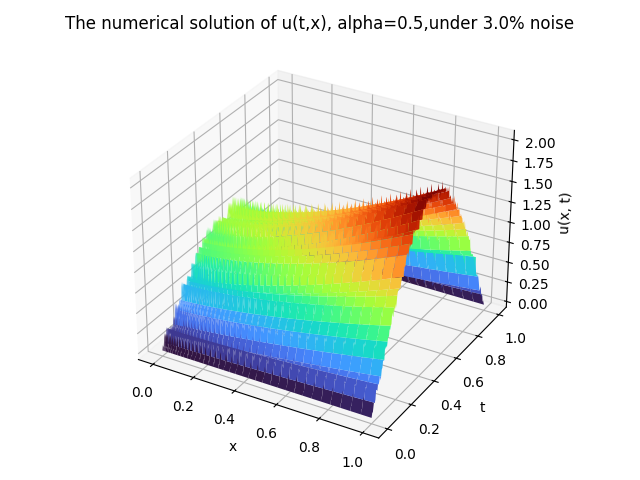}
  \end{minipage}
  \hfill
  \begin{minipage}{0.3\linewidth}
    \centering
    \includegraphics[width=\linewidth]{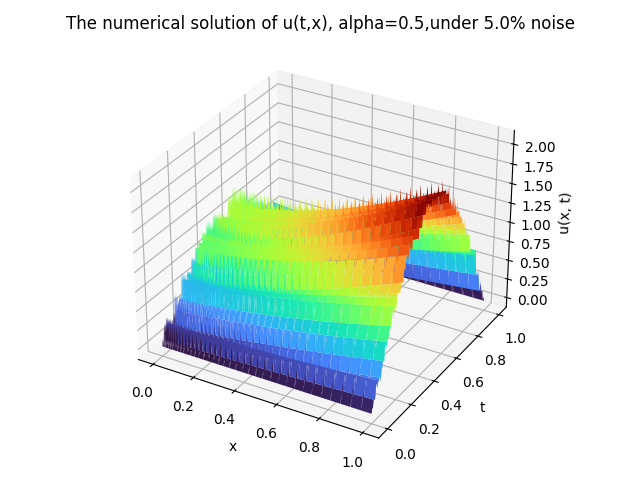}
  \end{minipage}

  \caption{Numerical reconstructions of $u(x,t)$ over the full time interval using noisy data:
  $1\%$ noise (left), $3\%$ noise (middle), and $5\%$ noise (right), for $\alpha=0.5$.}
  \label{fig3}
\end{figure}

Figures~\ref{fig:noise_pp} and \ref{fig3} complement these quantitative results. Even for moderate noise, the reconstructed profiles preserve the main shape of the exact initial state, while higher noise levels
primarily introduce smooth amplitude distortions rather than spurious oscillations. Overall, the experiments confirm that the proposed finite-difference-based forward solver combined with Tikhonov regularisation yields a numerically robust framework for recovering $u_0$ from noisy final-time data.

\section{Conclusion}\label{conc}
In this work, we have provided a comprehensive analysis of the backwards-in-time problem for a time-fractional pseudo-parabolic equation with time-dependent coefficients.
From a theoretical standpoint, we established the existence and
uniqueness of a classical solution by means of a spectral expansion and the theory of weakly singular Volterra integral equations. We also derived stability estimates showing the continuous dependence of the initial data on the final-time measurements. Numerical treatments were developed for both the forward and inverse problems. For the direct problem, we constructed a fully discrete finite difference scheme using the $L1$ approximation on graded meshes for the Caputo derivative. We rigorously proved that the scheme is unconditionally stable and convergent, achieving an optimal convergence rate.
Building upon this efficient forward solver, we formulated the inverse reconstruction as a minimisation problem regularised by the Tikhonov method.
Numerical experiments confirmed the theoretical findings, demonstrating that the proposed algorithm effectively recovers the unknown initial state even in the presence of noise in the terminal data.
Future research may extend this framework to multi-dimensional domains or nonlinear source terms, where the interplay between memory effects and nonlinearity poses further challenges.

\section*{Data availability}
Data will be made available on request.

\section*{Acknowledgements}
This work was supported by the Ministry of Science and Higher Education of the Republic of Kazakhstan (No. AP27508473).



\begin{thebibliography}{00}

\bibitem{BER_1989}
G.I. Barenblatt, V.M. Entov, V.M. Ryzhik,
\textit{Theory of Fluid Flows Through Natural Rocks},
Kluwer Academic Publishers, Dordrecht, 1990.

\bibitem{Pad_2004}
V. Padr\'on,
Effect of aggregation on population recovery modeled by a forward--backward
pseudoparabolic equation,
Trans. Am. Math. Soc. 356 (2004) 2739--2756.
\url{https://doi.org/10.1090/S0002-9947-03-03340-3}

\bibitem{Ting_1963}
T.W. Ting,
Certain non-steady flows of second-order fluids,
Arch. Ration. Mech. Anal. 14 (1963) 1--26.
\url{https://doi.org/10.1007/BF00250690}

\bibitem{BBM_1972}
T.B. Benjamin, J.L. Bona, J.J. Mahony,
Model equations for long waves in nonlinear dispersive systems,
Philos. Trans. R. Soc. Lond. A 272 (1972) 47--78.
\url{https://doi.org/10.1098/rsta.1972.0032}

\bibitem{Hui_1968}
R.R. Huilgol,
A second-order fluid of the differential type,
Int. J. Non-Linear Mech. 3 (1968) 471--482.
\url{https://doi.org/10.1016/0020-7462(68)90032-2}

\bibitem{BZK_1960}
G.I. Barenblatt, I.P. Zheltov, I.N. Kochina,
Basic concepts in the theory of seepage of homogeneous liquids in fissured
rocks,
J. Appl. Math. Mech. 24 (1960) 1286--1303.
\url{https://doi.org/10.1016/0021-8928(60)90107-6}

\bibitem{AKS_2011}
A.B. Al'shin, M.O. Korpusov, A.G. Sveshnikov,
\textit{Blow-up in Nonlinear Sobolev-Type Equations},
De Gruyter Series in Nonlinear Analysis and Applications, vol. 15,
Walter de Gruyter, Berlin, 2011.

\bibitem{ZVY_2010}
V.G. Zvyagin, M.V. Turbin,
Investigation of initial-boundary value problems for mathematical models of
the motion of Kelvin--Voigt fluids,
J. Math. Sci. 168 (2010) 157--308.
\url{https://doi.org/10.1007/s10958-010-9981-2}


\bibitem{AA_1997}
A. Asanov, E.R. Atamanov,
\textit{Nonclassical and Inverse Problems for Pseudoparabolic Equations},
Inverse and Ill-Posed Problems Series, vol. 7,
VSP, Utrecht, 1997.

\bibitem{LP_1997}
A. Lorenzi, E. Paparoni,
Identification problems for pseudoparabolic integrodifferential operator
equations,
J. Inverse Ill-Posed Probl. 5 (1997) 235--253.
\url{https://doi.org/10.1515/jiip.1997.5.3.235}

\bibitem{FU_2004}
V.E. Fedorov, A.V. Urasaeva,
An inverse problem for a linear Sobolev-type equation,
J. Inverse Ill-Posed Probl. 12 (2004) 387--395.
\url{https://doi.org/10.1515/1569394042248210}

\bibitem{LT_2011}
A.Sh. Lyubanova, A. Tani,
An inverse problem for a pseudoparabolic equation of filtration:
existence, uniqueness and regularity,
Appl. Anal. 90 (2011) 1557--1571.
\url{https://doi.org/10.1080/00036811.2010.530258}

\bibitem{Yaman_2012}
M. Yaman,
Blow-up solution and stability for an inverse problem for a pseudo-parabolic
equation,
J. Inequal. Appl. 2012 (2012) 274.
\url{https://doi.org/10.1186/1029-242X-2012-274}

\bibitem{BMSFT_2016}
M. Bertsch, F. Smarrazzo, A. Tesei,
Pseudo-parabolic regularization of forward--backward parabolic equations:
power-type nonlinearities,
J. Reine Angew. Math. 712 (2016) 51--80.
\url{https://doi.org/10.1515/crelle-2013-0123}

\bibitem{LV_2019}
A.Sh. Lyubanova, A.V. Velisevich,
Inverse problems for the stationary and pseudoparabolic equations of diffusion,
Appl. Anal. 98 (2019) 1997--2010.
\url{https://doi.org/10.1080/00036811.2018.1442001}

\bibitem{AAA_2022}
S.N. Antontsev, S.E. Aitzhanov, G.R. Ashurova,
An inverse problem for the pseudo-parabolic equation with $p$-Laplacian,
Evol. Equ. Control Theory 11 (2022) 399--414.
\url{https://doi.org/10.3934/eect.2021005}

\bibitem{KS_2022}
Kh. Khompysh, A.G. Shakir,
An inverse source problem for a nonlinear pseudoparabolic equation with
$p$-Laplacian diffusion and a damping term,
Quaest. Math. 46 (2023) 1889--1914.
\url{https://doi.org/10.2989/16073606.2022.2115951}

\bibitem{RSTT_2022}
M. Ruzhansky, D. Serikbaev, B.T. Torebek, N. Tokmagambetov,
Direct and inverse problems for time-fractional pseudo-parabolic equations,
Quaest. Math. 45 (2022) 1071--1089.
\url{https://doi.org/10.2989/16073606.2021.1928321}

\bibitem{Khompysh_2024}
Kh. Khompysh,
Determination of a time-dependent source in semilinear pseudoparabolic
equations with a Caputo fractional derivative,
Chaos Solitons Fractals 189 (2024) 115716.
\url{https://doi.org/10.1016/j.chaos.2024.115716}

\bibitem{KHSI_2024}
Kh. Khompysh, M.J. Huntul, M.K. Shazyndayeva, M.K. Iqbal,
An inverse problem for a pseudoparabolic equation: existence, uniqueness,
stability and numerical analysis,
Quaest. Math. 47 (2024) 1979--2001.
\url{https://doi.org/10.2989/16073606.2024.2347432}

\bibitem{BT_2024}
B. Bekbolat, N. Tokmagambetov,
Inverse source problem for the pseudoparabolic equation associated with the
Jacobi operator,
Bound. Value Probl. 2024 (2024) 52.
\url{https://doi.org/10.1186/s13661-024-01859-x}

\bibitem{Bockstal_2026}
K. Van Bockstal, Kh. Khompysh,
A time-dependent inverse source problem for a semilinear pseudo-parabolic
equation with Neumann boundary condition,
Comput. Math. Appl. 208 (2026) 97--112.
\url{https://doi.org/10.1016/j.camwa.2026.01.038}

\bibitem{Kilbas06}
A.A. Kilbas, H.M. Srivastava, J.J. Trujillo,
\textit{Theory and Applications of Fractional Differential Equations},
North-Holland Mathematics Studies, vol. 204,
Elsevier, Amsterdam, 2006.


\bibitem{Sakamoto_2011}
K. Sakamoto, M. Yamamoto,
Initial value/boundary value problems for fractional diffusion-wave equations
and applications to some inverse problems,
J. Math. Anal. Appl. 382 (2011) 426--447.

\bibitem{WZ_2018}
T. Wei, Y. Zhang,
The backward problem for a time-fractional diffusion-wave equation in a
bounded domain,
Comput. Math. Appl. 75 (2018) 3632--3648.

\bibitem{YZL_2020}
F. Yang, Y. Zhang, X. Li,
Landweber iterative method for identifying the initial value of a time-space
fractional diffusion-wave equation,
Numer. Algorithms 83 (2020) 1509--1530.

\bibitem{Floridia_2020}
G. Floridia, Z. Li, M. Yamamoto,
Well-posedness for backward problems in time for general time-fractional
diffusion equations,
Atti Accad. Naz. Lincei Rend. Lincei Mat. Appl. 31 (2020) 593--610.

\bibitem{Yamamoto_2020}
G. Floridia, M. Yamamoto,
Backward problems in time for a fractional diffusion-wave equation,
Inverse Probl. 36 (2020) 125016.

\bibitem{LKHC_2020}
N.H. Luc, D. Kumar, L.T.D. Hang, N.H. Can,
On a final value problem for a nonhomogeneous fractional pseudo-parabolic
equation,
Alexandria Eng. J. 59 (2020) 4353--4364.

\bibitem{ZVZ_2021}
Y. Zhang, T. Wei, Y.X. Zhang,
Simultaneous inversion of two initial values for a time-fractional
diffusion-wave equation,
Numer. Methods Partial Differ. Equ. 37 (2021) 24--43.

\bibitem{LZS21}
L.D. Long, Y. Zhou, R. Sakthivel, N.H. Tuan,
Well-posedness and ill-posedness results for a backward problem for a
fractional pseudo-parabolic equation,
J. Appl. Math. Comput. 67 (2021) 175--206.
\url{https://doi.org/10.1007/s12190-020-01488-4}

\bibitem{YXJ_2022}
F. Yang, J.M. Xu, X.X. Li,
Regularization methods for identifying the initial value of a time-fractional
pseudo-parabolic equation,
Calcolo 59 (2022) 47.
\url{https://doi.org/10.1007/s10092-022-00492-3}

\bibitem{DR_2023}
H. Di, W. Rong,
Regularized solution approximation of forward/backward problems for a
fractional pseudo-parabolic equation with random noise,
Acta Math. Sci. 43 (2023) 324--348.
\url{https://doi.org/10.1007/s10473-023-0118-3}

\bibitem{ST_2024}
D. Serikbaev, N. Tokmagambetov,
Determination of initial data in the time-fractional pseudo-hyperbolic
equation,
Symmetry 16 (2024) 1332.
\url{https://doi.org/10.3390/sym16101332}

\bibitem{Becker11}
L.C. Becker,
Resolvents and solutions of weakly singular linear Volterra integral
equations,
Nonlinear Anal. 74 (2011) 1892--1912.
\url{https://doi.org/10.1016/j.na.2010.10.060}

\bibitem{Stynes17}
M. Stynes, E. O'Riordan, J.L. Gracia,
Error analysis of a finite difference method on graded meshes for a
time-fractional diffusion equation,
SIAM J. Numer. Anal. 55 (2017) 1057--1079.
\url{https://doi.org/10.1137/16M1082329}

\bibitem{Altybay_26}
A. Altybay,
Numerical identification of a time-dependent coefficient in a
time-fractional diffusion equation with integral constraints,
Z. Angew. Math. Phys. 77 (2026) 41.
\url{https://doi.org/10.1007/s00033-025-02653-0}

\end{thebibliography}
\end{document}